\documentclass[12pt]{article}
\usepackage[utf8]{inputenc}
\usepackage[a4paper,top=3cm,bottom=3.5cm,left=3cm,right=3cm,marginparwidth=1.75cm]{geometry}

\usepackage{amssymb}
\usepackage{amsmath,amsfonts, amsthm}
\usepackage{tikz-cd, mathtools}
\usepackage{colortbl}
\usepackage{stix}
\usepackage{enumitem}
\usepackage{hyperref}
\usepackage[toc,page]{appendix}
\usepackage{lineno}
\usepackage{soul}
\usepackage{indentfirst}

\newtheorem{theorem}{Theorem}[section]
\newtheorem{proposition}[theorem]{Proposition}
\newtheorem{lemma}[theorem]{Lemma}
\newtheorem{corollary}[theorem]{Corollary}
\newtheorem{definition}[theorem]{Definition}
\newtheorem{example}[theorem]{Example}
\newtheorem{remark}[theorem]{Remark}

\DeclareMathOperator{\Ima}{Im}
\DeclareMathOperator{\Ker}{Ker}

\DeclareMathOperator{\Id}{Id}

\DeclareMathOperator{\rep}{Rep}

\newcommand{\dec}[1]{\lbrbrak #1\rbrbrak}

\newcommand{\Vect}{{\rm \bf Vect}}

\newcommand{\gen}[1]{{\langle #1\rangle}}
\newcommand{\biggen}[1]{{\big\langle #1\big\rangle}}

\newcommand{\cA}{\mathcal{A}}
\newcommand{\cB}{\mathcal{B}}

\newcommand{\cI}{\mathcal{I}}
\newcommand{\cJ}{\mathcal{J}}
\newcommand{\cK}{\mathcal{K}}
\newcommand{\cR}{\mathcal{R}}
\newcommand{\cS}{\mathcal{S}}
\newcommand{\cT}{\mathcal{T}}
\newcommand{\cM}{\mathcal{M}}
\newcommand{\cW}{\mathcal{W}}

\newcommand{\mtb}{{\bf B}}

\newcommand{\E}{{\rm \bf E}}
\newcommand{\R}{{\rm \bf R}}

\newcommand{\coker}{{\rm coker}}

\makeatletter
\newsavebox\myboxA
\newsavebox\myboxB
\newlength\mylenA
\newcommand{\sbm}[1]{{\let\amp=&\left[\begin{smallmatrix}#1\end{smallmatrix}\right]}}
\newcommand{\interval}[2]{{\lbrbrak #1 \mkern1mu, #2 \rbrbrak}}
\makeatother

\graphicspath{{figures/}}

\title{{Partial} Matchings Induced by
	Morphisms between Persistence Modules}
\date{}
\author{R. Gonzalez-Diaz, M. Soriano-Trigueros, A. Torras-Casas}

\begin{document}
	
	\maketitle
	\begin{center}
		\centering Departamento de Matemática Aplicada I, Universidad de Sevilla\\
		\{rogodi, msoriano4, atorras\}@us.es
	\end{center}

\begin{abstract}
We study how to obtain
partial matchings using the block function $\cM_f$, 
induced by
a morphism $f$ between persistence modules.
$\cM_f$ is defined algebraically and
is linear with respect to direct sums of morphisms.
We study some interesting properties of $\cM_f$,
and provide a way of obtaining $\cM_f$ using matrix operations.
\end{abstract}

\section{Introduction}

Persistent homology has become one of the most important tools in Topological Data Analysis \cite{Data, Computational}.
Persistence modules help to understand persistent homology~\cite{structure,categorification,OudotQuiver}.
Specifically, {\it persistence modules indexed by a 
totally ordered set $T$}
are functors from $T$ to the category
 $\Vect_k$, 
 where $k$ is
 a fixed field.
In this setting, 
under mild
assumptions,
a persistence module can be completely described by a multiset of intervals of $T$ called its 
{\it barcode}
\cite{decomposition}.

In practice,
TDA software take discrete data as an input,
and give
a barcode as the output 
\cite{gudhi, giottotda, scikittda}.
In some situations,
the user may want
to repeat the
procedure, 
after applying some minor modifications to the original data.
In such
case, 
two questions arise.
Could
we reuse the calculations
previously made to obtain the new barcode,
taking advantage of the similarities in the input?
Is there any relation between both barcodes?

Answering the first question would speed up the calculations significantly. 
Answering the second question would allow, for example,
to describe how intervals in the barcode change (or are kept unchanged) when the data is modified.
More concretely,
if a change in data induces a morphism $f: V  \pmb{\rightarrow} U$ between persistence modules indexed by $T$, then
answering the second question means
to know 
how $f$ induces
a partial matching $\rep\,\mtb(V) \mapsto
\rep\,\mtb(U)$ between representations of the corresponding barcodes.
It 
is known that 
such a partial matching
cannot be functorial \cite{induced1}.

Trying to answer both questions, we could think of
two possible research directions:
(1) Considering the morphism
$f$ as a persistence module in its own right and describing it in terms of ``simple pieces'',
that
may have an interpretation at the barcode level.
(2) Trying to define rules that produce
a partial matching induced by $f$, 
guaranteeing that it satisfies some desirable properties.
Before going into
details,
let us comment on
the state of the art in
both directions.

\paragraph{1. Decomposition of persistence modules.}
We say that a persistence module $V$ is decomposable when $V \simeq U \oplus W$ with $U, W \neq 0$.
Otherwise,
$V$ is said to be indecomposable.
Under mild assumptions, indecomposable modules indexed 
over $T$ are well-known and are called
interval modules \cite{structure,decomposition}.
Indecomposable  modules of the form
\begin{equation}\label{eq:ladderform}
\begin{tikzcd}
			U \\
			V \arrow[u, "f"]
		\end{tikzcd}
 		\simeq
\begin{tikzcd}
    U_1
    \arrow[r] & 
       U_2
       \arrow[r] & \ldots \arrow[r] & 
          U_n          \\
    V_1
    \arrow[r]\arrow[u] &   V_2
    \arrow[r]\arrow[u] & \ldots \arrow[r]
    & 
    V_n
    \arrow[u]
\end{tikzcd}
\end{equation}
where
$f:V\to U$ is a morphism between persistence modules,
are also well understood for all
$n \leq 4$.
When $n > 4$, the theory becomes increasingly complex,
and for $n \geq 6$ there is no way to parametrize the set of indecomposable modules since the underlying graph (the quiver) is of ``wild'' type (see, for example, 
\cite{realization,ladder,rectangle} for the use of quivers in TDA).
Recall that the category of 
modules of the form (\ref{eq:ladderform}), also known as ladder modules, is
isomorphic to the category of morphisms between persistence modules indexed by the set
${\bf n}=\{1, \ldots, n\}$ (see \cite{matrix}).

\paragraph{2. The induced partial
matching $\chi_f$.}
In \cite{induced1} and \cite{induced2}, 
given $f:V
\rightarrow U$,
the authors provided 
a 
procedure to construct a partial matching
between representations of the barcodes, $\rep\,\mtb(V)$ and $\rep\,\mtb(U)$, denoted by $\chi_f$.
The aim of 
providing such a construction
was to give an explicit proof of the \emph{Stability Theorem} for barcodes \cite{induced1}.
However, this partial matching has some limitations when applied to real data.
In particular, it produces the following undesired result.

Consider the morphism $f$ of persistence modules
determined by the following commutative diagram:
\begin{equation}\label{eq:direct sum}
		\begin{tikzcd}
			U \\
			V \arrow[u, "f"]
		\end{tikzcd}
		\simeq 
		\begin{tikzcd}
			0 \arrow[r] & 0 \arrow[r] & 0 \\
			0 \arrow[r]  \arrow[u, swap]& k \arrow[r, "\Id"] \arrow[u, swap]& k \arrow[u]
		\end{tikzcd}
		\oplus
		\begin{tikzcd}[/tikz/column 3/.style={column sep=-0.5em}]
			k \arrow[r, "\Id" ] & k \arrow[r] & 0 \\
			0 \arrow[r]  \arrow[u]& k \arrow[r] \arrow[u, "\Id", swap]& 0 \arrow[u]
			& .
		\end{tikzcd}
	\end{equation}
	Then $\rep\,\mtb(V) = \{[2,3]_{1}, [2,2]_{1} \}$ and $\rep\,\mtb(U
	) = \{ [1,2]_{1} \}$.
	Note that, looking at the decomposition of $f$,
	one would expect to obtain the following matching:
	\[
	    [2,3]_{1} \longmapsto \emptyset\ , 
	    \qquad 
	    [2,2]_{1} \longmapsto [1,2]_{1}\ .
	\]
	However, $\chi_f$ produces the following
partial matching:
	\[
	    [2,3]_{1} \overset{\chi_f}{\longmapsto} [1,2]_{1}\ , 
	    \qquad 
	    [2,2]_{1} \overset{\chi_f}{\longmapsto} \emptyset
	\]
which is counter-intuitive.

In opposition to $\chi_f$,
we propose $\cM_f$, 
a {\it block} function induced by the morphism $f$, 
which gives a better insight of how $f$ works.
As we will see in Example~\ref{ex:22->12}:
\[
    \cM_f([2,2], [1,2]) = 1\ , \qquad \cM_f([2,3]_{1}, [1,2]) = 0\ .
\]
More concretely, $\cM_f$ 
 is defined entirely algebraically, is linear with respect to the 
direct sum of morphisms
and can be easily calculated using matrix column reductions.
As explained in Subsection~\ref{sec:inducing}, $\cM_f$ also allows to induce partial matchings efficiently.

The paper is organized as follows.
Firstly, the background is introduced in Section~\ref{sec:preliminaries}.
In Section~\ref{sec:operators-im-ker}, we introduce the operators $\Ima^{\pm}$ and $\Ker^{\pm}$ and illustrate them by using \emph{persistence bases}.
In Section~\ref{sec:block},
the aforementioned operators are used to
define the block function 
$\cM_f$ induced by a morphism $f:V\to U$ and we prove it is well-defined.
Its main properties,
together with
an explanation of how $\cM_f$ can be used to compute partial matchings, are given in Section~\ref{sec:properties}.
In Section~\ref{sec:matrix}, we give 
a method to compute
$\cM_f$ 
using matrix calculations.
Finally,
the main conclusions and 
some open questions are 
discussed
in Section~\ref{sec:future}.

\section{Preliminaries}\label{sec:preliminaries}

We recall 
the notions of persistence modules,
decorated endpoints, 
barcodes and partial matchings.
We also introduce some necessary algebraic tools.

\subsection{Persistence modules}\label{sec:persistencemodules}

All vector spaces considered in this paper  are defined over a fixed field $k$ with unit denoted by $1_k$ and they can be infinite dimensional, except in Section~\ref{sec:matrix}, where they 
 are finite dimensional. 
Vectors 
 are expressed in column form.

A persistence module $V$ indexed by a totally ordered set $T$
is a functor from 
$T$
to $\Vect_k$.
Then $V$
consists of 
a set of 
vector spaces $V_p$
for
$p\in T$ and a set of
linear maps $\rho_{pq}: V_p
\rightarrow 
V_q$ for
$p \leq q$ satisfying that $\rho_{ql}\rho_{pq} = \rho_{pl}$ if $p \leq q \leq l$; and $\rho_{pp}$ being the identity map.
The set of linear maps $\{\rho_{pq}\}_{p\leq q}$ is denoted  by $\rho$ and its elements are known as the \emph{structure maps} of $V$.
Since the category of persistence modules is abelian, the direct sum of persistence modules together with the intersection and quotient of persistence 
modules
are also persistence modules
(see for example~\cite[Sec.~2]{categorification}).

We consider the functor category of persistence modules.
In other words,
given two persistence modules, $V$ and $U$, with structure maps, $\rho$ and $\phi$,
a morphism, $f:V\to U$, is given by a set of linear maps $\{ f_{p}
\}_{p
\in T}$,
such that $f_{q}
\rho_{pq}
= 
\phi_{pq}
f_{p}
$ if
$p\leq q$.
A morphism $f$ is \emph{injective} (\emph{surjective}) if all its linear maps $ f_{p}$,
$p
\in T$,
are injective (surjective).
Notice that $\Ima f$ and $\Ker f$ are particular cases of persistence modules.
We assume that all persistence modules that appear in this paper satisfy the \emph{descending chain condition (d.c.c.) for images and kernels.}
In other words,
for $t \geq p_1 \geq p_2 \geq \ldots$ and $t \leq \ldots \leq  q_2 \leq q_1 $,
the following chains stabilize:
\[
V_t \supseteq \Ima \rho_{p_1 t} \supseteq \Ima \rho_{p_2 t} \supseteq \ldots
\qquad \mbox{ and }\qquad 
V_t \supseteq \Ker \rho_{t q_1} \supseteq \Ker \rho_{t q_2} \supseteq \ldots
\]

\subsection{Decorated points and interval modules}

In this subsection, 
we
use the 
notation appearing in \cite{induced1} based on the one introduced 
in \cite[Sec.~2.4]{structure}.
Let $\E$ denote the set of {\it decorated endpoints} defined as
$\E:=\R \times D\cup \{ -\infty, \infty\}$,
where $D=\{-,+\}$.
In what follows, 
decorated points
$(r,-)$ and $(r,+)$  are denoted by
$r^-$ and $r^+$, respectively.
Note that $\E$ can be seen as a totally ordered set stating that
$r^- < r^+$ together with the order inherited by the extended reals.
The sum $+:
\E \times \R \rightarrow \E$  
is defined as   
$r^{\pm} + s := (r + s)^{\pm}$.
There is a bijection between the pairs $\{ (a,b) \in {\E \times \E} : a < b \}$ and the intervals of $\R$.
The following table shows all possible cases:
\[
    \begin{tabular}{c|ccc}
    &$s^-$&$s^+$&$\infty$\\
    \hline
    $-\infty$&$(-\infty,s)$&$(-\infty,s]$&$(-\infty,\infty)$\\
    $r^-$&$[r,s)$&$[r,s]$&$[r,\infty)$\\
    $r^+$&$(r,s)$&$(r,s]$&$(r,\infty)$\\
    \end{tabular}
\]

From now on,
an interval of $\R$ represented by $(a,b) \in \E \times \E$, with $a<b$, is denoted by $\interval{a}{b}$. 
Intervals of 
any other totally ordered set
may turn up.
We use the letters $a,b,c$ and $d$ to denote  elements of $\E$; the letters $r,s$ and $t$ to denote elements of $\R$; and the letters $p$, $q$ 
and $l$ 
to denote elements of a general, totally ordered set $T$.

Given an interval $I$ of $\R$,
the \emph{interval module}, $k_I$, is composed by $k_{It}= k$ for all $t \in I$ and $k_{It}= 0$ otherwise, while the structure maps are given by the identity whenever possible.
As shown in the next subsection,
interval modules are the building blocks of persistence modules. 

\subsection{Decomposition
 of persistence modules}\label{subsec:decom}

The results
that appear in this section
are directly taken from \cite{decomposition}
where the following statement is proven.

\begin{theorem}[Theorem~1.2 of \cite{decomposition}]\label{the:decomposition}
For any persistence module $V$ indexed by $\R$ satisfying the d.c.c. for images and kernels,
we have:
	\[ V \simeq  \bigoplus_{I\in S_V} 
	\big(\oplus_{m_I} 
	k_I\big) \]
where $S_V$ is a set of intervals of $\R$ and $m_I$ is the multiplicity of $k_I$.
\end{theorem}

The proof of the theorem uses the operators $\Ima^{\pm}$ and $\Ker^{\pm}$ (which are introduced in Section~\ref{sec:operators-im-ker}) as well as the concept of \emph{section}.
A \emph{section} of a vector space 
$A$ 
is a pair of vector spaces 
$(F^-,F^+)$
such that $F^- \hookrightarrow F^+ \hookrightarrow 
A$. We say that a set $\{ (F^-_\lambda,F^+_\lambda) : \lambda\in \Lambda \}$ of sections (with $\Lambda$ its index set)
of
$A$ 
\emph{is disjoint} if, for all $\lambda\neq \mu$,
either $F^+_\lambda \hookrightarrow F^-_\mu$ or $F^+_\mu \hookrightarrow F^-_\lambda$.

Sections are used in Lemma~\ref{lem:BandV} and Theorem~\ref{the:nnested}. In particular, we use the following result,
whose justification
is part of 
the proof of Theorem~6.1 in \cite{decomposition}.
\begin{lemma}\label{lem:subsumcond}
	Suppose that $\{(F^-_\lambda,F^+_\lambda) : \lambda\in \Lambda \}$ is a disjoint set of
	sections of a vector space $A$.
    Then
	\[
	    \bigoplus_{\lambda\in\Lambda} 
	    \big( F^+_\lambda\, \big/ \,F^-_\lambda\big)
	   \,
	   \hookrightarrow \,A\ .
    \]
\end{lemma}
\subsection{Barcodes and partial matchings}

A \emph{multiset} is a pair $(S,m)$ where $S$ is a set and $m: S \rightarrow \mathbb{N} \cup \{\infty\}
$ represents the multiplicity of 
the elements of $S$.
An element of the multiset $(S,m)$  is denoted by the pair $(I, m_I
)$ where $I
\in S$ and $m_I = m(I)$.
The \emph{barcode} of a persistence module $V$ is the multiset 
$\mtb (V) = (S_V, m)$ where 
$S_V$ is the set of 
intervals that appear in 
the decomposition
of $V$,
and
$m_I$ is
the multiplicity of $I\in S$.
The \emph{representation of a multiset} $(S,m)$ is the set
\[
\rep(S,m) = \{(I,i
) \in S \times \mathbb{N} : i
\leq m_I
) \} \, .
\]
From now on, we 
use the notation $I_i
$ instead of $(
I,i
)$.

\begin{example}
    Consider the persistence module 
    \[
    U\simeq
    \begin{tikzcd}[row sep=large, column sep = large]
		k^2
		\arrow[r, "\sbm{ 1 \amp 0 \\ 0 \amp 1 \\ 0 \amp 0}"] & k^3
		\arrow[r, "\sbm{ 0 \amp 0 \amp 1 }"] & k
	\end{tikzcd}
    \]
    that can be decomposed as
        \[
        k_{[1,2]} \oplus k_{[1,2]} \oplus k_{[2,3]}\ .
    \]
    Then
    its barcode is $\mtb(U) = \{ ([1,2],2), ([2,3],1) \}$
    and the
    representation of its barcode is
    $\rep \mtb(U)$ $=$ $\{ [1,2]_1, [1,2]_2, [2,3]_1\} $, which can be displayed as:
    \begin{center}
    \includegraphics[scale=0.33]{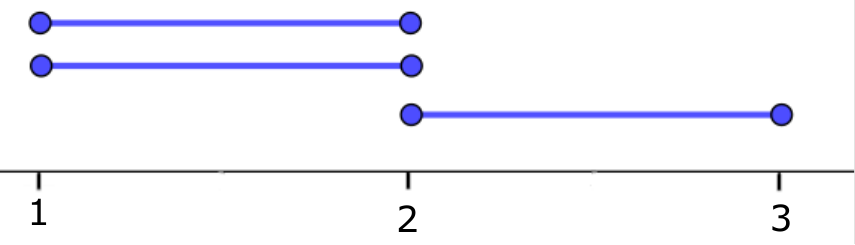}
    \end{center}
\end{example}

Given two barcodes, $\mtb_1$ and $\mtb_2$,
a \emph{partial matching}
between $\rep \mtb_1$ and $\rep \mtb_2$
is
a bijection
$\sigma: R_{1} \rightarrow
R_{2}$ where
$R_1 \subset \rep \mtb_1$ and $R_2 \subset \rep \mtb_2$.
By abuse of notation,
we might write instead $\sigma: \rep \mtb_1 \rightarrow
\rep \mtb_{2}$, and say $\sigma(I) = \emptyset$ when we mean $I \notin R_1$.
\begin{definition}
A \emph{block function}
between
two barcodes $\mtb_{1} = (S_1, m)$ and $\mtb_{2} = (S_2, n)$
is a function $\cM : S_1 \times S_2 \longrightarrow \mathbb{Z}_{\geq 0} \cup \{\infty \}$ such that:
\[
\sum_{J \in S_2} \cM(I, J) \leq m_I\ .
\]
\end{definition}
\begin{remark}\label{rem:partial}
Note that when a block function satisfies
\[
    \sum_{I \in S_1} \cM(I, J) \leq n_J\ ,
\]
it is straightforward to show that $\cM$ induces a partial matching between $\rep \mtb_1$ and $\rep \mtb_2$.
\end{remark}

\subsection{
\texorpdfstring{Persistence Bases}{Persistence Bases}
}

A persistence 
 basis \cite{distributed, structure} for a persistence module $V$ is an isomorphism 
\[
    \alpha:
    \bigoplus_{i\in \Gamma}\, k_{\interval{a_i}{b_i}} \rightarrow V\ ,
\]
where
$\Gamma$ is an index set. By Theorem~\ref{the:decomposition}, such persistence bases exist for all persistence modules satisfying the d.c.c. condition for images and kernels. The \emph{persistence generator} $\alpha_i:k_{\interval{a_i}{b_i}} \rightarrow V$ is defined as the 
morphism 
$\alpha$ restricted to $k_{\interval{a_i}{b_i}}$ for $i \in \Gamma$.
When we write $\alpha_i \sim {\interval{a_i}{b_i}}$, we mean that $ k_{{\interval{a_i}{b_i}}}$ is the domain of $\alpha_i$.
We also specify a persistence basis $\alpha$ by its set of persistence generators, $\cA=\{\alpha_i
\}_{i \in \Gamma}$ and we denote its cardinality by $\#\cA$.

\begin{definition}\label{def:persistence-basis}
Given a subset $\cS = \{ \alpha_i\}_{i \in \Lambda}$ of $\cA$, we define the span  of $\cS$, denoted by $\gen{\cS}$, as the image of the sum of persistence generators of $\cS$, that is  
\[
\gen{\cS}=\Ima\left(\bigoplus_{i \in \Lambda} \alpha_i: \bigoplus_{i\in \Lambda}\, k_{\interval{a_i}{b_i}} \rightarrow V \right)\ .
\]
For $t \in \R$, we define $\cS_t \coloneqq \{\alpha_{it}^1 : i \in \Lambda \mbox{ and }  t\in \interval{a_i}{b_i}\}$ where $\alpha_{it}^1$ means $\alpha_{it}(1_k)$ by abuse of notation. In particular, $ V_t =  \gen{\cA_t}$ and $ \gen{\cS}_t =  \gen{\cS_t}$, where $\cA_t$ and $\cS_t$ are linearly independent sets of vectors in $V_t$.
Notice that in general, the submodule $\gen{\cS}$ of
$V$ depends on $\alpha$.
\end{definition}

The following result is later used to prove that $\cM_f$ is well-defined.
  \begin{lemma}\label{lem:subbasis}
  Let $V$ be a submodule of
    a persistence module $U$ indexed by $\R$ such that $V_t = 0$ for all $t\in \dec{d, \infty}$
    for some $d\in\E$.
    Then, given a persistence basis
    $\alpha$ for
    $U$,
    we have that 
     $V$ is also a submodule of
     \[
       W= \ 
       \gen{\alpha_i : 
       \alpha_i \sim \dec{\cdot, b'} \text{ with } b' \leq d}\ . 
     \]
    \begin{proof}
        By contradiction, assume 
        that $V$ is not a submodule of $W$.
          Then, for some $s\in\dec{-\infty,d}$,
        there exists $x\in V_s$ such that $x \not\in W_s$. Thus,
        \[
            x = \sum_{i\in \Gamma_1
            } x_i\alpha_{is}^1 + \sum_{j\in\Gamma_2} x_{j}\alpha_{js}^1
        \]
        for coefficients $x_i,x_j\in k\setminus \{0\}$ for all $i \in \Gamma_1$ and $j \in \Gamma_2$; where $\Gamma_1$ and $\Gamma_2$ are two disjoint indexing subsets $\Gamma_1 \subseteq \{ i : \alpha_i \sim \dec{\cdot, b'} \text{ with } b'\leq d
        \}$ and $\Gamma_2 \subseteq \{ j:
        \alpha_{j} \sim \dec{\cdot, b'} \text{ with } d<b' \}$.
          Note that $\Gamma_2$ is non empty since 
        $x \not\in W_s$ by hypothesis. We choose $j'$ in $\Gamma_2$,
        such that $\alpha_{j's}^1 \neq 0$ and $x_{j'} \neq 0$.
        Denote the right endpoint of $\alpha_{j'}$ as $d'$.
        Let $t\in\dec{d,d'}$.
        Then
        \[
            0 = \rho_{st} x = \sum_{j\in\Gamma_2'} x_j\alpha_{jt}^1\ ,
        \]
        where $\Gamma_2' = \{ j \in \Gamma_2 : \rho_{st}(\alpha_{js}^1) \neq 0\}$, so that $j' \in \Gamma'_2$.
        However, as $\{ \alpha_{jt}^1\}_{j \in \Gamma'_2}$ is linearly independent by  Definition~\ref{def:persistence-basis},
         $x_j$ must be zero for all $j \in \Gamma_2'$,
        including $j'$,
        leading
        to a contradiction.
    \end{proof}
\end{lemma}

In particular, notice that the submodule $W$ of
$V$ from Lemma~\ref{lem:subbasis} is independent for the chosen persistence basis $\alpha$, as given another persistence basis $\alpha'$ with the corresponding submodule $W'$, we would obtain $W\subseteq W'$ and also $W \supseteq  W'$; thus $W=W'$.  

Finally, let $f:V\rightarrow W$ be a 
morphism between persistence modules and let 
$\cA$ and $\cB$ be persistence bases for $V$ and $U$, respectively.
If $f$ is an injection, then $\#\cA\leq \#\cB$ while if $f$ is a projection, then $\#\cA\geq \#\cB$. This is a consequence of the result of persistence submodules and quotients given in~\cite[Thm.~4.2]{induced1}.

\subsection{Direct Limits}

We need the notion of direct limits to define the block
function ${\cal M}_f$ and the notion of sections to prove it is well-defined.
The definition of direct limit and some useful lemmas involving direct limits
are given in~\ref{app:direct}.
In this section, we just give a characterization of direct limits for our context.

\begin{proposition}[Characterization of direct limits for persistence modules]~\label{prop:direct-limits}
    Let V be a persistence module indexed by an interval $\dec{a,b}$, and consider the endpoint $a <d \leq b$.
    Then 
    \[
        \varinjlim_{t \in \dec{a,d}}V_t\simeq \dfrac{\bigoplus_{p\in \dec{a,d}} V_p}{Z}
    \]
    where $Z$ is the vector space generated by $v_p \oplus -\rho_{pq}(v_p) \in V_p \oplus V_q $, with $p\leq q$ and 
       $p,q\in\dec{a,d}$.
\end{proposition}

\begin{proof}
    See \cite[Def.~3.41]{structure} and \cite[Prop.~3.43]{structure}.
\end{proof}

The direct limit $\varinjlim_{t \in \dec{a,d}}V_t$ does not depend on $a$ and, intuitively, it is isomorphic to the vector space generated by the intervals 
$\interval{c}{b}$ with $c \leq d \leq b$.
Actually, if $d$ represents a closed right endpoint,  i.e. $d=s^+$, it follows from the characterization that
$    \varinjlim_{t \in \interval{a}{s^+}} V_t  = V_s\ .$

\begin{example}
Consider a
persistence module $V$ isomorphic to
\[
  k_{[1,3)} \oplus k_{(1, 3]} \oplus k_{[0, 4)} 
\]
and a persistence basis 
${\cal A}=\{\alpha_1,\alpha_2,\alpha_3\}$
for $V$.
For example, if $v \in V_2$, then
$v = x_1 \alpha_{12}^1 + x_2 \alpha_{22}^1 + x_3 \alpha_{32}^1\; $ 
for some $x_1, x_2, x_3 \in k$.
    Note that for $ 1^+ < a < b \leq 3^-$ and $t,s \in \dec{a,b}$, all structure maps $\rho_{st}$ are injective.
    Then,
       \[
        Z = \gen{\alpha_{it}^1 \oplus-\alpha_{is}^1 \, : \, i=1,2,3;\, t\leq s \in \dec{a,b}}.
    \]
    Applying Proposition~\ref{prop:direct-limits},
    for a fixed $i$, all $\alpha_{it}^1$ represent the same class in $\underset{t \in \interval{a}{b}}{\varinjlim}V_t$.
    Then, denoting this class by $\alpha^1_{i}$,
    we have
    \[
        \varinjlim_{t\in \interval{a}{b}} V_t = \gen{\alpha^1_{i} : i=1,2,3}\,.
    \]
    On the contrary, for $b=3^+$, we have that $(\alpha_{1t}^{1} \oplus 0)  \in Z$ for  $t \in \interval{a}{3^-}$ since  $\rho_{t3} \alpha^1_{1t} = 0$,  then
    \[
        \varinjlim_{t\in \interval{a}{3^+}} V_t = V_3 = \gen{\alpha^1_{i3} : i=2,3}\ .
    \]
\end{example}

\section{The operators \texorpdfstring{$\Ima^{\pm}$}{Im+-} and \texorpdfstring{$\Ker^{\pm}$}{Ker+-}}
\label{sec:operators-im-ker}


The operators $\Ima^{\pm}$ and $\Ker^{\pm}$  were used in~\cite{decomposition} to prove Theorem~\ref{the:decomposition}. In this section, we introduce these operators from the point of view of persistence bases with the aim of using them later to construct $\cM_f$.
Let $V$ be a persistence module indexed by $\R$ 
with structure maps $\rho$ and
with persistence
basis $\cA = \{\alpha_i\}_{i\in \Gamma}\,$. Let $c\in\E$. We consider the following subsets of $\cA$:
\begin{itemize}
  \item $\cI^+_{c}(\cA) = \left\{\alpha_i \in \cA \mid \alpha_i \sim \dec{a_i, b_i} \mbox{ such that } a_i \leq c \right\}$\ ,
  \item $\cI^-_{c}(\cA) = \left\{ \alpha_i \in \cA \mid \alpha_i \sim \dec{a_i, b_i} \mbox{ such that } a_i < c \right\}$\ ,
  \item $\cK^+_{c}(\cA) = \left\{ \alpha_i \in \cA \mid \alpha_i \sim \dec{a_i, b_i} \mbox{ such that } b_i \leq c \right\}$\ ,
  \item $\cK^-_{c}(\cA) = \left\{ \alpha_i \in \cA \mid \alpha_i \sim \dec{a_i, b_i} \mbox{ such that } b_i < c \right\}$\ .
\end{itemize}

As a convention, given
$t \in \R$, we write $\cI^{\pm}_{ct}(\cA)$ instead of $\cI^{\pm}_{c}(\cA)_t$ and $\cK^{\pm}_{ct}(\cA)$ instead of $\cK^{\pm}_{c}(\cA)_t$.
As Lemma~\ref{lem:Im-Ker-plus-minus} shows, $\cI^{\pm}_{ct}(\cA)$ and $\cK^{\pm}_{ct}(\cA)$ generate the following vector spaces, which were 
first introduced
in~\cite{decomposition}:
\[
	\Ima^+_{ct}(V) := \bigcap_{\substack{s\in 
	\interval{c}{t^+}
	}
	} \Ima \rho_{st}\ ,
	\quad  \Ima^-_{ct}(V) := \bigcup_{\substack{s\in \interval{-\infty}{c}} \\ } \Ima \rho_{st}\ , 
\quad\mbox{for $t \in  \interval{c}{\infty}$;}
\]
\[ 
	\Ker^+_{ct}(V) := \bigcap_{\substack{r\in \interval{c}{\infty}}} \Ker \rho_{tr}\ ,
	\quad
	\Ker^-_{ct}(V) :=\bigcup_{\substack{r\in 
	\interval{t^-}{c}
}}\Ker \rho_{tr}\ ,
\quad\mbox{for $t \in  \interval{-\infty}{c}$.}
\]
By convention,
$\Ima^-_{ct}(V) := 0$ if $c = -\infty$, and $\Ker^+_{ct}(V) := V(t)$ if $c = \infty$. 

\begin{lemma}\label{lem:Im-Ker-plus-minus}
  For all $c \in \E$, we have the equalities:
\begin{enumerate}[label=\emph{\alph*})]
\item\label{Im-plus-minus}   
$\Ima^\pm_{ct}(V) = \gen{\cI^\pm_{ct}(\cA)}$ for all $t \in \dec{c, \infty}$,
\item\label{Ker-plus-minus}  $\Ker^\pm_{ct}(V) = \gen{\cK^\pm_{ct}(\cA)}$ for all $t \in \dec{-\infty, c}$.
\end{enumerate}
\end{lemma}

\begin{proof}
    See the first proof in Appendix~\ref{app:persistence-basis}
\end{proof}

\begin{example}\label{ex:ima-ker}
Consider a persistence module $V$ isomorphic to $k_{[1,2]} \oplus k_{[2, 3]}$ and a persistence basis $\cA = \{ \alpha_1 , \alpha_2\}$ for $V$ where the generators are given by
    \[
        \alpha_1 = 
		\begin{tikzcd}
			k \arrow[r, "\sbm{1 \\ 0 }" ] & k^2
			\arrow[r, "\sbm{ 0 \, 1 }"] & k \\
			k \arrow[r] \arrow[u, "\Id"] & k \arrow[r, "\Id"] \arrow[u, "\sbm{ 1 \\ 0 }"]& 0 \arrow[u]
		\end{tikzcd} 
		\qquad\mbox{and}\qquad
        \alpha_2 = 
		\begin{tikzcd}[/tikz/column 3/.style={column sep=-0.5em}]
			k \arrow[r, "\sbm{1 \\ 0 }" ] & k^2
			\arrow[r, "\sbm{ 0 \, 1 }"] & k \\
			0 \arrow[r] \arrow[u] & k \arrow[r, "\Id"] \arrow[u, "\sbm{ 0 \\ 1 }"]& k \arrow[u, "\Id"] & .
		\end{tikzcd}
    \]
    Using Lemma~\ref{lem:Im-Ker-plus-minus}, the subspaces 
    $\Ima_{22}^\pm(V)$ and $\Ker_{32}^\pm(V)$ of $V_2=k^2$ are given by
    \[
        \gen{\cI^+_{22}(\cA)} = \gen{\cK^+_{32}(\cA)} = \gen{\alpha_{12}^1, \alpha_{22}^1 } = k^2 \, , \quad
        \gen{\cI^-_{22}(\cA)} = \gen{\cK^-_{32}(\cA)} = \gen{\alpha^1_{12}} =         \langle 
            \begin{smallmatrix}
                1 \\ 
                0 
            \end{smallmatrix} 
        \rangle.
    \]
\end{example}
As shown in the following result, $\Ima^{\pm}$ and $\Ker^{\pm}$ are linear.
\begin{proposition}\label{prop:ima}
Let $c\in \E$ and let $V^1$ and $V^2$ be two persistence modules. Then,
 \[
 \Ima^\pm_{ct}(V^1 \oplus V^2) =  \Ima^\pm_{ct}(V^1) \oplus \Ima^\pm_{ct}(V^2) \ \mbox{ for all } t \in \dec{c, \infty}\,\mbox{ and}
\]
\[
 \Ker^\pm_{ct}(V^1 \oplus V^2) =  \Ker^\pm_{ct}(V^1) \oplus \Ker^\pm_{ct}(V^2)\ \mbox{ for all } t \in \dec{-\infty, c}\,.
\]
\end{proposition}
\begin{proof}
First, notice that we have the following equalities
\[
    \Ima^-_{ct}(V) = \varinjlim_{s \in \dec{-\infty, c}} \Ima \rho_{st}
    \qquad \Ker^-_{ct}(V) = \varinjlim_{r\in \dec{t^- , c}} \Ker \rho_{tr},
\]
and since direct sums commute with limits (see Lemma~\ref{app:commute-direct-sum}), the result holds for these operators.
In the other case, by the d.c.c. for images and kernels, we know there exists $s \in \dec{c,t^+}$ and $r \in \dec{c,\infty}$ such that,
\[
    \Ima^+_{ct}(V) = \Ima \rho_{st}\,, \qquad 
    \Ima^+_{ct}(V^1) = \Ima \rho_{st}|_{V^1}\,, \qquad \Ima^+_{ct}(V^2) = \Ima \rho_{st}|_{V^2}\,,
\]
\[
    \Ker^+_{ct}(V) = \Ker \rho_{t r}\,, \qquad 
    \Ker^+_{ct}(V^1) = \Ker \rho_{t r}|_{V^1}\,, \qquad \Ker^+_{ct}(V^2) = \Ker \rho_{t r}|_{V^2}\,,
\]
see \cite[Lemma~2.1]{decomposition} for details.
In such cases, the linearity follows directly from the linearity of $\Ima$ and $\Ker$.
\end{proof}

Let $V$ be a persistent module and let $I=\dec{a,b}$ be an interval of $\R$. We use the operators $\Ima^{\pm}$ and $\Ker^{\pm}$ to define the persistence modules $V_I^+, V_I^-$ and $V_I$ pointwisely.
For $t \in I$, define
\begin{align*}
	&V^+_{It} := \Ima^+_{at}(V) \cap \Ker^+_{bt}(V)
	\\
	&V^-_{It} := \Ima^-_{at}(V) \cap \Ker^+_{bt}(V) + \Ima^+_{at}(V) \cap \Ker^-_{bt}(V)
	\\
	&V_{It} := V^+_{It} \, \big/ \,V^-_{It}
\end{align*}
and, for $t\not\in I$, let them all be $0$. 
These modules played an important role in Theorem~\ref{the:decomposition}.
Specifically, $V_{I}$ is isomorphic to the direct sum $\oplus_{m_I} k_I$.
As we will see  later, such modules are also
a key ingredient for the definition of $\cM_f$.
We end this section with an interpretation of  $V_{I}
$ in terms of persistence bases.
Define 
\begin{align*}\nonumber
\cA^+_{I
} = \ & \cI^+_a(\cA) \cap \cK^+_b(\cA) \\
\cA^-_{I
} = \ & (\cI^+_a(\cA) \cap \cK^-_b(\cA)) \cup (\cI^-_a(\cA) \cap \cK^+_b(\cA))\ ,
\end{align*}
which can also be described as:
\begin{align*}\nonumber
\cA^+_{I
} = & \left\{ \alpha_i \in \cA \mid a_i \leq a \mbox{ and } b_i \leq b \right\}
\\ 
\cA^-_{I
} = & \left\{ \alpha_i \in \cA \mid (a_i < a \mbox{ and } b_i \leq b) \mbox{ or } (a_i \leq a \mbox{ and } b_i < b) \right\}\ .
\end{align*}
We define $\cA_{I
t}$  as  $\cA^+_{I
t} \setminus \cA^-_{I
t}$ and obtain:
\[
\cA_{I}
 = \left\{ \alpha_i \in \cA \mid (a_i = a \mbox{ and } b_i = b) \right\}\ .
\]

\begin{proposition}\label{prop:A-plus-minus}
$V_{I
t}^+ = \biggen{\cA^+_{I
t}}$ and
$V_{I
t}^- = \biggen{\cA^-_{I
t}}$ for all $t \in I
$.
\end{proposition}
\begin{proof}
Let $I=\dec{a, b}$ with $a,b\in \E$.
From Lemma~\ref{lem:Im-Ker-plus-minus} and Lemma~\ref{lem:intersection-sub-bases},
we have:
\[
V_{I
t}^+ =
\Ima^+_{at}(V) \cap \Ker^+_{bt}(V) =
\gen{\cI^+_{at}(\cA)} \cap \gen{\cK^+_{bt}(\cA)} = \gen{\cI^+_{at}(\cA) \cap \cK^+_{bt}(\cA)} =
\gen{\cA^+_{I
t}} \ .
\]
And, similarly,
\begin{multline}\nonumber
V_{I
t}^- =
\Ima^-_{at}(V) \cap \Ker^+_{bt}(V) + \Ima^+_{at}(V) \cap \Ker^-_{bt}(V) \\
=\gen{\cI^-_{at}(\cA)} \cap \gen{\cK^+_{bt}(\cA)} + \gen{\cI^+_{at}(\cA)} \cap \gen{\cK^-_{bt}(\cA)}
\\ = \gen{\cI^-_{at}(\cA) \cap \cK^+_{bt}(\cA)} + \gen{\cI^+_{at}(\cA) \cap \cK^-_{bt}(\cA)} =
\gen{\cA^-_{I
t}} \ . \tag*{\qedhere}
\end{multline}
\end{proof}

We are now ready  to present an interpretation of $V_{I}$ in terms of persistence bases.

\begin{theorem}\label{th:V-quotient-base}
$V_{I
t} \simeq \gen{\cA_{I
t}}$ for all $t \in I
$.
\end{theorem}

\begin{proof}
Using Proposition~\ref{prop:A-plus-minus} and Lemma~\ref{lem:quotient-sub-base}, we obtain:
\begin{equation*}
V_{I
t} = \dfrac{V^+_{I
t}}{V^-_{I
t}} = 
\dfrac{\gen{\cA^+_{I
t}}}{\gen{\cA^-_{I
t}}} \simeq 
\gen{\cA^+_{I
t} \setminus \cA^-_{I
t}} =
\gen{\cA_{I
t}}\ . \qedhere
\end{equation*}
\end{proof}

\begin{example}
    Going back to Example~\ref{ex:ima-ker}, we know that the interval $[2,3]$ has multiplicity $1$.
    Then,
    $V_{[2, 3]2}$ must be a space of dimension 1.
    Using Theorem~\ref{th:V-quotient-base}, this can be rapidly checked by the following computation:
    \[
    V_{[2,3]2} \simeq \gen{\cA_{[2,3]2}} = \gen{\cA^+_{[2,3]2} \setminus \cA^-_{[2,3]2}} =  \gen{\{\alpha_{12}^1, \alpha_{22}^1\} \setminus \{\alpha_{12}^1 \} } = \gen{\alpha_{22}^1} \, .
    \]
\end{example}

\section{The block function \texorpdfstring{$\cM_f$}{Mf}
}\label{sec:block}

The definition of $\cM_f$  is formulated 
algebraically,
via operators of persistence modules.
Recall from Subsection~\ref{subsec:decom} that
given a persistence module $V$, the multiplicity $m_I$ of an interval $I$ in the barcode $ \mtb(V)$
is completely determined by the persistence module 
$V_I$ defined in Section~\ref{sec:operators-im-ker}.
Given a morphism $f:V\to U$ between persistence modules, our aim 
is to create a new persistence module, $X_{IJ}$, relating $V_I$ and $U_J$ via $f$. Hence, denoting the respective barcodes of $V$ and $U$ by $\mtb(V)=(S_V, m)$ and $\mtb(U)=(S_U, n)$, from $X_{IJ}$, we obtain the block function $\cM_f$ relating $(I,m_I)$ and $(J, n_J)$ for all pairs of intervals $I\in S_V$ and $J \in S_U$.

For this, let us assume that $V$ and $U$, with structure maps $\rho$ and $ \phi$, respectively, are indexed by $\R$.
Let 
$I = \dec{a,b}$ and $J = \dec{c,d}$.
For
$t \in I\cap J$, define the vector space:
\[
	X_{IJt}
	:= \dfrac{ fV_{It}^+ \cap U_{Jt}^+ }{ fV_{It}^- \cap U_{Jt}^+ +  fV_{It}^+ \cap U_{Jt}^-}\ .
\]
If $t \notin I \cap J$, then $X_{IJt} := 0$.  Notice that when
we write $fV^\pm_{It}$ we mean $f_t(V^\pm_{It})$.
Observe that,
since 
$X_{IJ}$ is
made up of sums, intersections and quotient of persistence modules,
$X_{IJ}$ is  also a persistence module. Intuitively, $X_{IJt}$ is equal to the intersection $f(V_{It})\cap U_{Jt}$. Inspecting such an intersection on the limit, we obtain a way to relate the elements $(I,m_I)$ and $(J,m_J)$ which is induced by $f$.

\begin{definition}
Let $\mtb(V)= (S_V, m)$ and $\mtb(U) = (S_U, n)$.
We define the function
$\cM_f:S_V\times S_U\to \mathbb{Z}_{\geq 0}\cup \{\infty \}$
as,
\[ 
\cM_f(I,J) := \dim \varinjlim_{t\in I \cap J} X_{IJt} \ .
\]
\end{definition}

\begin{example}\label{ex:induced}
    Consider the morphism between persistence modules
    $f:V\rightarrow U$
    given by the following 
    commutative diagram:
	\[
		\begin{tikzcd}
			U \\
			V \arrow[u, "f"]
		\end{tikzcd}
		\simeq 
		\begin{tikzcd}[/tikz/column 3/.style={column sep=-0.5em}]
			k \arrow[r, "\sbm{1 \\ 0 }" ] & k^2
			\arrow[r, "\sbm{ 0 \, 1 }"] & k \\
			0 \arrow[r] \arrow[u]& k \arrow[r, "\Id"] \arrow[u, "\sbm{ 1 \\ 1 }"]& k \arrow[u, "\Id"] & .
		\end{tikzcd}
	\]
	The 	
	respective barcodes are
    \[        \mtb(V) = \{ ([2,3],1) \}\, \qquad \mbox{ and }\qquad \mtb(U) = \{([1,2], 1), ([2,3],1) \}\,.
    \]
    We 
    calculate
    the set $f\cA = \{ f \alpha_1 \}$ which will be used later.
    Concatenating both functions,     $\alpha_1$
    and $f$, we get
    \[
        f \alpha_1 = k_{[2,3]} \overset{Id}{\rightarrow} V \overset{f}{\rightarrow} U =
	f\,.
    \]
    If $I = [2,3]$ and $J = [1,2]$,
    then we can calculate $X_{IJ}$ and $X_{II}$ as follows.
    First,
    \[
        fV^+_{I2} \cap U_{J2}^+ 
        = 
        \langle 
            \begin{smallmatrix}
                1 \\ 
                1 
            \end{smallmatrix} 
        \rangle
        \cap
        \langle 
            \begin{smallmatrix}
                1 \\ 
                0 
            \end{smallmatrix} 
        \rangle
        =
        0 \;\mbox{ and }\;X_{IJ2} = 0\, .
    \]
    Moreover
    \[
        fV^+_{I2} \cap U_{I2}^+ 
        = 
        \langle 
            \begin{smallmatrix}
                1 \\ 
                1 
            \end{smallmatrix} 
        \rangle
        \cap
        \langle 
            \begin{smallmatrix}
                1 & 0\\ 
                0 & 1
            \end{smallmatrix} 
        \rangle
        =
        \langle
            \begin{smallmatrix}
                1 \\ 
                1 
            \end{smallmatrix}
        \rangle\, ,
    \]
    \[
        fV_{I2}^-\cap U_{I2}^+ +  fV_{I2}^+\cap U_{I2}^-
        =
        \langle 
            \begin{smallmatrix}
                0 \\ 
                0 
            \end{smallmatrix} 
        \rangle
        \cap
        \langle 
            \begin{smallmatrix}
                1 & 0\\ 
                0 & 1
            \end{smallmatrix} 
        \rangle
        +
        \langle
            \begin{smallmatrix}
                1 \\ 
                1 
            \end{smallmatrix} 
        \rangle
        \cap
        \langle 
            \begin{smallmatrix}
                1 \\ 
                0 
            \end{smallmatrix} 
        \rangle
        =
        0\, ;
    \]
    and
    \[
        fV^+_{I3} \cap U_{I3}^+ 
        = 
        \langle 
            1
        \rangle
        \cap
        \langle 
            1
        \rangle
        =
        \langle 
            1
        \rangle\, ,
    \]
    \[
        fV_{I3}^- \cap  U_{I3}^+ +  fV_{I3}^+\cap U_{I3}^-
        = 
        \langle 
            0
        \rangle
        \cap
        \langle 
            1
        \rangle
        +
        \langle 
            1
        \rangle
        \cap
        \langle 
            0
        \rangle
        =
        \langle 
            0
        \rangle\, .
    \]
    From these computations, we conclude that
    the only non-zero case for $X$ 
     is the submodule
    $ X_{[2,3][2,3]}: \,0 {\longrightarrow}         
    \langle
        \begin{smallmatrix}
            1 \\ 
            1 
        \end{smallmatrix}
    \rangle
    \overset{\sbm{ 0 \, 1 }}{\longrightarrow} k
    $,
    and then 
    \[
    {\cM}_f([2,3],[1,2])=0\;\mbox{ and }{\cM}_f([2,3],[2,3])=1\ .
    \]
\end{example}

As the following result shows,
the cases of study can be reduced considerably.
\begin{proposition}\label{pro:cadb}
Let $I = \dec{a,b}$ and $J = \dec{c,d}$.
 If $\cM_f\left( I,J
 \right) \neq 0\,,$ then
 \[
 c \leq a \leq d \leq b\ .
 \]
\end{proposition}
\begin{proof}
    By definition, $a<b$ and $c < d$. Moreover, if $I \cap J$
    is empty, then
    $
    \cM_f\left( I,J
    \right)
    $
    is 
    zero.
    Then we just have to prove that $\cM_f\left(I,J
    \right)$ is zero when $a<c$ or $b<d$.
    Following the definitions of $\Ima^{\pm}$ and $\Ker^{\pm}$ given in Section~\ref{sec:operators-im-ker}, we have:
    \begin{align*}
        f_{t}(\Ima_{at}^+(V)) \hookrightarrow\Ima_{at}^+(U) & \qquad  f_{t}(\Ker^+_{bt}(V)) \hookrightarrow  \Ker^+_{bt}(U)\ ,
    \end{align*}
    and:
    \begin{align*}
        \Ima_{at}^+(U) \hookrightarrow  \Ima_{ct}^-(U) & \qquad  \Ker^+_{bt}(U) \hookrightarrow  \Ker^-_{dt}(U)\ ,
    \end{align*}
    when $a<c$ or $b<d$ respectively (see  Lemma~7.1 in \cite{decomposition} for details).
    If $a < c$,
    \begin{equation}\label{eq:ac}
        fV_{It}^+ \cap U_{Jt}^+ \hookrightarrow   f_{t}(\Ima_{at}^+(V)) \cap  \Ker^+_{dt}(U) \hookrightarrow  \Ima_{ct}^-(U) \cap \Ker^+_{dt}(U) \hookrightarrow  U_{Jt}^-\,,
    \end{equation}
    then $fV_{It}^+ \cap U_{Jt}^+ \hookrightarrow  fV_{It}^+ \cap U_{Jt}^- $
    and
    $X_{{IJt}}    = 0$.
    Exchanging $\Ima^{\pm}$ by $\Ker^{\pm}$ in Expresion~(\ref{eq:ac}), the same reasoning works when $b<d$.
\end{proof}

This is an expected result due to the commutativity of $f$ with respect to the structure maps.
An equivalent result for $\chi_f$  is given in \cite[Prop.~5.3]{induced1}.

\subsection{\texorpdfstring{$\cM_f$}{Mf} is well-defined}

Given a morphism $f:V\to U$ between persistence modules $V$ and $U$ with barcodes $\mtb(V)=(S_V,m)$ and $\mtb(U)=(S_U,n)$, respectively,
the aim of this subsection is to prove the following theorem.
\begin{theorem}\label{the:well}
    $\cM_f$ is well-defined, that is, it is a block function.
\end{theorem}

Given $I=\dec{a,b}\in S_V$, we have to prove that $\sum_{J\in S_U}\cM_f(I,J)\leq m_I$. 
Fixing $J=\dec{c,d}\in S_U$,
we can assume that
$d \leq b$  
without loss of generality (see Proposition~\ref{pro:cadb}).
In such case, $I \cap J$ has $d$ as the right endpoint.

To prove Theorem~\ref{the:well}, we would like to relate the vector spaces $X_{IJt}$ with $fV_{It}$ and use that $\dim fV_{It} \leq m_I$ (since by definition $fV_{It}$ means $fV^+_{It}/fV^-_{It}$).
However, 
the vector spaces $X_{IJt}$ and 
$fV_{It}$
are subspaces of
different vector spaces and there is no straightforward way of relating them.
For this reason,
we define a new intermediate vector space, and use it to compare the dimension of $\bigoplus_{J\in S_U} X_{IJt}$ and $fV_{It}$.
Lemma~\ref{lem:BandV} makes this statement precise. To prove it,
for fixed $I\in S_V$ and $c,d\in \E$, we use the following persistence modules indexed by $t\in\R$:
 \[
    A^d_{ct}
    :=
    \dfrac{
    fV_{It}^- \cap U^+_{\dec{c,d} t}
    +
    fV_{It}^+ \cap U^-_{\dec{c,d} t}
    }{
    fV_{It}^- \cap U^+_{\dec{c,d} t}
    }
    \quad \text{ and, } \quad
    B^d_{ct}
    :=
    \dfrac{
    fV_{It}^+ \cap U^+_{\dec{c,d} t}
    }{
    fV_{It}^- \cap U^+_{\dec{c,d} t}
    }\ .
\]
The choice of the vector spaces is justified by the following property, 
\[ 
    \dfrac{B^d_{ct}}{A^d_{ct}} \simeq  X_{IJt}\ .
\]
Due to our previous assumption about $I \cap J$, the following limits are equivalent:
\[
    \tilde A^d_c
    :=
    \varinjlim_{t \in I \cap J} A^d_{ct} 
    =
    \varinjlim_{t \in \dec{-\infty, d}} A^d_{ct}
    \quad \text{ and } \quad
    \tilde B^d_c
    :=
    \varinjlim_{t \in I \cap J} B^d_{ct} 
    =
    \varinjlim_{t \in \dec{-\infty, d}} B^d_{ct}\ .
\]
In the following, we write either $I\cap J$ or $\dec{-\infty, d}$ when taking limits, since it does not affect the calculation.
Due to Lemma~\ref{app:commute-quotients}, we also have:
\[
    \dfrac{\tilde B^d_{c}}{\tilde A^d_{c}} \simeq \varinjlim_{t \in I \cap J} X_{IJt}\ .
\]
Moreover, varying $c$, we can see the vector 
spaces 
$\tilde A^d_{c}$ and $\tilde B^d_{c}$ as persistence modules indexed by $c\in\E\,.$

\begin{lemma}\label{lem:BandV}
    For a fixed $d\in\E \,,$ we have:
    \begin{align*}
        \left( \bigoplus_{c<d} \varinjlim_{t \in I \cap J} X_{I\dec{c,d}t} \right) \hookrightarrow
        \varinjlim_{c
        < d} \tilde B^d_c \hookrightarrow  \varinjlim_{t \in \dec{-\infty, d}} fV_{It}\ .
    \end{align*}
\end{lemma}

\begin{proof}
First, let us prove
       the second injection.
   Notice
   that, for each $t \in I \cap J$,
    \[
        fV_{It}^+ \cap U_{Jt}^+ 
        \hookrightarrow 
        fV_{It}^+ 
        \quad \text{ and } \quad
         (fV_{It}^+ \cap U_{Jt}^+) \cap fV_{It}^- 
         = 
         fV_{It}^- \cap U_{Jt}^+\ ,
    \]
    so that $B_{ct}^d \hookrightarrow fV_{It}$.
    By Lemma~\ref{lem:short},
    we also have that, for each $c$,
    \[
        \tilde B_{c}^d 
        = 
        \varinjlim_{t \in \dec{-\infty, d}} B_{ct}^d
        \hookrightarrow 
        \varinjlim_{t \in \dec{-\infty, d}} fV_{It}\ .
    \]
    Using Lemma~\ref{lem:short} again
    and since $fV_{It}$ does not depend on $c$,
        \[
        \varinjlim_{c < d} \tilde B_{c}^d 
        \hookrightarrow 
        \varinjlim_{c < d} \left( \varinjlim_{t \in \dec{-\infty, d}} fV_{It} \right)
        =
        \varinjlim_{t \in \dec{-\infty, d}} fV_{It}\ .
    \]
    To obtain the other injection, let us prove that 
    $\big\{(\tilde A_{c}^d, \tilde B_c^d):c\in\E\big\}$ is a disjoint set of
    sections of $\varinjlim_{c < d} \tilde B_{c}^d$.
    In other words,
    let us prove that
    \begin{itemize}
        \item for each $c$, 
            \,
            $
                \tilde A_{c}^d 
                \hookrightarrow 
                \tilde B_{c}^d 
                \hookrightarrow 
                \underset{c < d}{\varinjlim} 
                \,\tilde B_{c}^d
            $
            \,
            and,
        \item for all $c' < c$,
            \,
            $
                \tilde B_{c'}^d 
                \hookrightarrow 
                \tilde A_{c}^d 
            $.
    \end{itemize}
    If so, the result  follows by Lemma~\ref{lem:subsumcond} since
    $\tilde B_c^d/\tilde A_c^d$ is isomorphic to $\varinjlim_{t \in I\cap J} X_{IJt}$.
    \\
    Note that $A_{ct}^d \hookrightarrow B_{ct}^d$ by definition.
    Then, by Lemma~\ref{lem:short}, $\tilde A_c^d \hookrightarrow\tilde B_c^d$.
    We also have that
    $\Ima_{c't}^+(U) \hookrightarrow  \Ima_{ct}^-(U)$ if $c' < c$, implying,
    \[
        U^+_{\dec{c',d}t} 
        = 
        \Ima_{c't}^+(U) \cap \Ker_{dt}^+(U) 
        \hookrightarrow
        \Ima_{ct}^-(U) \cap \Ker_{dt}^+(U) 
        \hookrightarrow
        U^-_{\dec{c,d}t}
    \]
    Then, $\big(fV_{It}^+ \cap U^+_{\dec{c',d}t}\big) \cap \big(fV_{It}^- \cap U^+_{\dec{c,d}t}\big) = fV_{It}^- \cap U^+_{\dec{c',d}t}$ and also
    \[
        fV_{It}^+ \cap U^+_{\dec{c',d}t} 
        \hookrightarrow 
        fV_{It}^+ \cap U^-_{\dec{c,d}t}
        \hookrightarrow 
        fV_{It}^+ \cap U^-_{\dec{c,d}t} + fV_{It}^- \cap U^+_{\dec{c,d}t}\ ,
    \]
    so that
    $
        B_{c't}^d
        \hookrightarrow 
        A_{ct}^d 
    $
    and using Lemma~\ref{lem:short} we obtain the inclusion
    $
        \tilde B_{c'}^d
        \hookrightarrow 
        \tilde A_{c}^d 
    $.
    \\
    We still have to prove that,
    \,
    $
        \tilde B_{c}^d 
        \hookrightarrow 
        \underset{c < d}{\varinjlim} 
        \tilde B_{c}^d.
    $
    We already proved that
    $\tilde B_{c'}^d 
    \hookrightarrow 
    \tilde A_{c}^d
    \hookrightarrow 
    \tilde B_{c}^d$
    for each $c' < c$.
    In particular,
    $\tilde B_{c'}^d 
    \hookrightarrow 
    \tilde B_{c}^d$ 
    for each $c' < c$.
   Finally, applying
   Lemma~\ref{lem:monos} we
   obtain the desired
   result.
\end{proof}
\begin{proof}[Proof of Theorem~\ref{the:well}]
We need to prove that $\sum_{J\in S_U} \cM_f(I,J) \leq m_I
$.
Using that $\cM_f(I,J) \neq 0$ for
$I=\dec{a,b}$ and $J=\dec{c,d}$ only if 
$c\leq a\leq d\leq b $,
we can rewrite the sum as
\[
    \sum_{J\in S_U} \cM_f(I,J) 
    = 
    \sum_{d \leq b}\sum_{c < d} \cM_f(I,J) 
    =
    \sum_{d \leq b}\sum_{c < d} \dim \varinjlim_{t \in \dec{-\infty, d}} X_{I \dec{c,d}t} 
\]
which is equivalent to
\[
    \sum_{d \leq b} \dim \left( \bigoplus_{c < d} \varinjlim_{t \in \dec{-\infty, d}} X_{I \dec{c,d}t} \right)
\]
 and, by Lemma~\ref{lem:BandV}, is less or equal to 
\[
     \sum_{d \leq b} \dim \, \varinjlim_{c < d} \tilde B_c^d\ .
\]
Moreover,
given a persistence basis $\cA$ of $fV_I$, we obtain the inequality
\begin{equation}\label{eq:bas}
    \dim \, \varinjlim_{c < d} \tilde B_c^d \leq \#\{ \alpha_i \in \cA: \alpha_i \sim \dec{\cdot, d} \}\ ,    
\end{equation} 
which is deduced later on. 
Then, as we have $V_I \twoheadrightarrow fV_I$,
\begin{equation*}
    \sum_{J\in S_U} \cM_f(I,J)
    \leq
    \sum_{d \leq b} \#\{ \alpha_i \in \cA: \alpha_i \sim \dec{\cdot, d} \} \leq \#\cA \leq m_I\ .
\end{equation*}
    To prove  Inequality~(\ref{eq:bas}),
    recall from the proof of Lemma~\ref{lem:BandV} that $B_{ct}^d$ is a submodule of $fV_{It}$.
   Since $B_{ct}^d = 0$ for $t\in\dec{d,\infty}$, by Lemma~\ref{lem:subbasis}, we have:
   \[
        B_{ct}^d \hookrightarrow \gen{\alpha^1_{it} : \alpha_i \sim \dec{\cdot, b'} \text{ with } b'\leq d }
   \]
   for $t\in\dec{c,d}$.
     By Lemma~\ref{lem:short}, direct limits are compatible with injections and so
  \[
         \varinjlim_{c < d} \tilde B_{ct}^d 
         = 
         \varinjlim_{c < d} \varinjlim_{t \in \dec{-\infty, d}} B_{ct}^d
         \hookrightarrow 
         \varinjlim_{c < d} \varinjlim_{t \in \dec{-\infty, d}} 
         \gen{\alpha^1_{it} : \alpha_i \sim \dec{\cdot, b'} \text{ with } b'\leq d }
   \]
   which is equal to
   \[
    \varinjlim_{t \in \dec{-\infty, d}} 
         \gen{\alpha^1_{it} : \alpha_i \sim \dec{\cdot, b'} \text{ with } b'\leq d }\ ,
    \]
   since it does not depend on $c$.
   Finally, since 
   the class of 
   $\alpha_{it}^1$ is zero in the above limit  if $\alpha_i \sim \dec{\cdot, b'}$ with $b'<d$, we have:
    \begin{align*}
    \dim & \varinjlim_{t \in \dec{-\infty, d}}
         \gen{\alpha^1_{it} : \alpha_i \sim \dec{\cdot, b'} \text{ with } b'\leq d }
         \\
     &\leq
     \dim \varinjlim_{t \in \dec{-\infty, d}} \gen{\alpha^1_{it} : \alpha_i \sim \dec{\cdot, b'} \text{ with } b' = d }.
    \end{align*}
    Then
    \[
        \dim \varinjlim_{c < d} \tilde B_{ct}^d  \leq \# \{\alpha_i : \alpha_i \sim \dec{\cdot, d}\}\ ,
    \]
    which is the same inequality as  (\ref{eq:bas}).
\end{proof}

\section{Properties of \texorpdfstring{$\cM_f$}{Mf}}\label{sec:properties}

Let us study the main properties of $\cM_f$ to obtain a better insight of how it works.
We already saw in Proposition~\ref{pro:cadb} that $\cM_f(\dec{a,b},\dec{c,d})$ is  non-zero only if $c \leq a \leq d \leq b$.
Another important property is its linearity.
\begin{theorem}\label{the:linear}
Given a direct sum of morphisms:
\[
   f^1 \oplus f^2 : V^1 \oplus V^2 \longrightarrow U^1 \oplus U^2
\]
We have that,
\begin{align*}
    X_{IJt}[f^1 \oplus f^2] &=  X_{IJt}[f^1] \oplus X_{IJt}[f^2] \\
    \text{ and }\,
       \cM_{f^1 \oplus f^2}(I,J) &= \cM_{f^1}(I,J) + \cM_{f^2}(I,J).
\end{align*}
\end{theorem}
\noindent
Since direct sums commute with quotients, finite intersections, finite sums and direct sums of persistence modules (see Appendix~\ref{app:direct}), Theorem~\ref{the:linear} is a direct consequence of Proposition~\ref{prop:ima}.

\begin{example}\label{ex:22->12}
    Theorem~\ref{the:linear} allows performing a quick calculation of $\cM_f$ for the morphism $f=f^1\oplus f^2$ given in Expression ~(\ref{eq:direct sum})
    and recalled here:
    \begin{equation*}
		\begin{tikzcd}
			U \\
			V \arrow[u, "f"]
		\end{tikzcd}
		\simeq 
		\begin{tikzcd}
			0 \arrow[r] & 0 \arrow[r] & 0 \\
			0 \arrow[r]  \arrow[u, swap] & k \arrow[r, "\Id"] \arrow[u, swap] & k \arrow[u]
		\end{tikzcd}
		\oplus
		\begin{tikzcd}[/tikz/column 3/.style={column sep=-0.5em}]
			k \arrow[r, "\Id" ] & k \arrow[r] & 0 \\
			0 \arrow[r]  \arrow[u]& k \arrow[r] \arrow[u, "\Id", swap]& 0 \arrow[u] & .
		\end{tikzcd}
	\end{equation*}
	Indeed, $\cM_{f^1}$ is always zero for the first summand, and it is 
	non-zero for the second one only  in the case $\cM_{f^2}([2,2],[1,2])=1$, as expected.
\end{example}

The following proposition explains how  $\cM_f({I,J})$ behaves
when the domain of $f$ is an interval module.
\begin{proposition}\label{pro:intervalToV}
    Consider a non-null morphism $f : k_I \rightarrow U$ and  a persistence basis $\beta$ for $U$.
    Let $\mathcal{J}$ be
    the  set of intervals 
    that appear in the expression of $f k_I$ in terms of $\beta$.
    Then
    $\cM_f(I,J) = 1$ if $J$ is the interval with the smallest length between the ones with the largest right endpoint in $\mathcal{J}$.
\end{proposition}

\begin{proof}
    Since $f$ is non-null,
    there exists a finite, non-empty set $\{\beta_i\}_{i = 1 \ldots n}$ and $t\in I$ such that
    $ f_t(1_k) = \sum_{i=1}^n x_i
    \beta_{it}^1$, for some $x_1,\dots,x_n\in k\setminus\{0\}$. Notice that the finiteness condition follows from  considering $\beta^{-1}_t f_t(1_k)$
    which can only have finite non-zero coordinates, since it lies in a direct sum of vector spaces.
\\
    Now, sort in decreasing order
    the set $\{ b' : \exists i \text{ with } \beta_i \sim \interval{\cdot}{b'}\}$ and define $b, d$ as the first and second value respectively.
    Then
    there is a subset $\Lambda \subset \{1, \ldots n\}$,
    such that, for $i \in \Lambda$, $\beta_i \sim \interval{\cdot}{b}$.
    We have that
    $f_s(1_k) = \sum_{i\in\Lambda} x_i
    \beta_{is}^1$ for  any $s \in \interval{d}{b}$.
    Now, let 
    $J$
    be the interval with the 
    smallest length in
    $\{ J'
    : \exists i\in \Lambda,\, \beta_i \sim 
    J'\} $.
    Then,
    by definition of $U^\pm_{J}$,
    we have that for $s \in \interval{d}{b}$, 
    $f_s(1_k)\in U^+_{Js}$
    but $f_s(1_k) \notin U^-_{Js}$.
    Then
    there is $K$ with 
    ${\interval{d}{b}} \subset K\subset I \cap J$ 
    such that
    $k_{K} \simeq X_{IJ}$.
    This implies that $\cM_{f}(I,J) = 1$,
    and by Theorem~\ref{the:well} it must be the only non-zero case.
\end{proof}

However, as the following example shows, 
$\cM_f$ does not always induce a unique partial matching,
since if there are nested bars, $\sum_{I \in S_V} \cM_{f}(I,J)$ can be greater than $n_J$
(see Remark~\ref{rem:partial}).

\begin{example}\label{ex:notmatch}
Consider the morphism of persistence modules
given by
the following commutative diagram:
\begin{equation*}
		\begin{tikzcd}
			U \\
			V \arrow[u, "f"]
		\end{tikzcd}
		\simeq 
		\begin{tikzcd}[/tikz/column 3/.style={column sep=-0.5em}]
			k \arrow[r] & k \arrow[r] & 0 \\
			k \arrow[r, "\sbm{1 \\ 0 }"]  \arrow[u] & k^2 \arrow[r, "\sbm{ 1 \; 0 }"] \arrow[u, swap, "\sbm{ 1 \; 1 }"] & k \arrow[u] & .
		\end{tikzcd}
	\end{equation*}
	Notice that $\mtb(V) = \{([1,3],1), ([2,2],1)\}$ and $\mtb(U) = \{([1,2],1)\}$.
	Let $I=[1,3]$, $K=[2,2]$ and $J=[1,2]$.
	Then
	\begin{align*}
	    fV_{I}^+ \cap U^+_{J} &= k \rightarrow k \rightarrow 0, \\
	    fV_{I}^- \cap U^+_{J} + fV_{I}^+ \cap U^-_{J} &= 0 \rightarrow 0 \rightarrow 0, \\
	    fV_{K}^+ \cap U^+_{J} &= 0 \rightarrow k \rightarrow 0, \\
	    fV_{K}^- \cap U^+_{J} + fV_{K}^+ \cap U^-_{J} &= 0 \rightarrow 0 \rightarrow 0 .
	\end{align*}
	Taking the direct limit in $d=2^+$, we have
	\[
	    \cM_f(I,J) = 1 \quad \text{ and,} \quad \cM_f(K,J) = 1.
	\]
	Then 
	$\cM_f(I,J)+\cM_f(K,J)>n_J=1$.
\end{example}

As the following theorem shows, if there are no nested bars then the block function $\cM_{f}$ induces a partial matching.

\begin{theorem}\label{the:nnested}
  Consider an ordered set of intervals, $\{ \dec{a^i, b^i} \}_{i \in \Gamma}\,,$ and an interval $J = \dec{c,d}$, such that for all $i \leq j$, $a^i \leq a^j$ and $b^i \leq b^j$, and $\sup_{i \in \Gamma} \{ a^i \} < d$.
  Then  \[
    \sum_{i \in \Gamma} \cM_{f}(\dec{a^i, b^i}, J) \leq n_J\ .
  \]
\end{theorem}

\begin{proof}
   Notice that we can assume without loss of generality that $c \leq a^i \leq d \leq b^i$ for all $i \in \Gamma$, since otherwise $X_{\dec{a^i, b^i} Jt} = 0$.
    Our strategy is similar to the proof of Theorem~\ref{the:well}.
      We define a set of subspace pairs of $U_{Jt}$, $ \{(A_{\dec{a^i, b^i}t}, B_{\dec{a^i, b^i}t}):i\in\Gamma\}$, for all $t\in\dec{e,d}$ where $e = \sup_{i \in \Gamma} \{ a^i \}$.
    We prove that it is a disjoint set of
    sections of $U_{Jt}$
    and that $B_{\dec{a^i, b^i}t}/ A_{\dec{a^i, b^i}t} \simeq X_{\dec{a^i, b^i} J t}$.
    Then, by Lemma~\ref{lem:subsumcond},
    \[
        \bigoplus_{i \in \Gamma} X_{\dec{a^i, b^i} Jt} \hookrightarrow U_{Jt}\ .
    \]
   As a consequence, using Lemma~\ref{app:commute-direct-sum} and Lemma~\ref{lem:short}, we obtain the inclusion
    \[
        \varinjlim_{t \in \dec{e,d}} \bigoplus_{i \in \Gamma} X_{\dec{a^i, b^i} J
        t} = \bigoplus_{i \in \Gamma} \varinjlim_{t \in \dec{e,d}} X_{\dec{a^i, b^i} J
        t}  {\hookrightarrow}  
        \, \varinjlim_{t \in \dec{e,d}}U_{Jt} \ ,
    \]
    Thus,
    \[
        \sum_{i \in \Gamma}\cM_{f}(\dec{a^i, b^i}, J) = 
        \dim \left( \bigoplus_{i \in \Gamma} \varinjlim_{t \in \dec{e,d}} X_{\dec{a^i, b^i} J
        t} \right) \leq 
        \dim \, \varinjlim_{t \in \dec{e,d}} U_{Jt} = n_J.
    \]
    We define $A_{\dec{a^i, b^i} t}$ and $B_{\dec{a^i, b^i} t}$ as follows:
    \[
        A_{\dec{a^i, b^i} t} := \dfrac{fV_{\dec{a^i, b^i} t}^- \cap U_{Jt}^+}{fV_{\dec{a^i, b^i} t}^- \cap U_{Jt}^-}\ , \qquad  B_{\dec{a^i, b^i} t} := \dfrac{fV_{\dec{a^i, b^i} t}^+ \cap U_{Jt}^+}{fV_{\dec{a^i, b^i} t}^+ \cap U_{Jt}^-}\ .
    \]
    Notice that $A_{\dec{a^i, b^i} t} \hookrightarrow B_{\dec{a^i, b^i} t}$ since $(fV_{\dec{a^i, b^i} t}^- \cap U_{Jt}^+) \cap (fV_{\dec{a^i, b^i} t}^+ \cap U_{Jt}^-) = fV_{\dec{a^i, b^i} t}^- \cap U_{Jt}^-$.
    To prove that 
    $\{(A_{\dec{a^i, b^i}t}, B_{\dec{a^i, b^i}t}):i\in\Gamma\}$  is a disjoint  set of
    sections 
    of $U_{Jt}$\,,
    let us see that
    $B_{\dec{a^i, b^i} t} {\hookrightarrow} U_{Jt}$ and
    $B_{\dec{a^i, b^i} t} {\hookrightarrow} A_{\dec{a^j, b^j} t}$ whenever $i < j$.
    The first injection is true since
    \[
        fV_{\dec{a^i, b^i} t}^+ \cap U_{Jt}^+ \hookrightarrow  U_{Jt}^+ \quad \text{  and } \quad fV_{\dec{a^i, b^i} t}^+ \cap U_{Jt}^+ \cap U_{Jt}^- = fV_{\dec{a^i, b^i} t}^+ \cap  U_{Jt}^-\ .
    \]
        Recall that the intervals in the set $\{\dec{a^i, b^i}\}_{i \in \Gamma}$ are all different, which implies that for $i < j$, either $a^i < a^j$ or $b^i < b^j$. Consequently, we have that $V_{\dec{a^i, b^i} t}^+ \hookrightarrow V_{\dec{a^j, b^j} t}^-$ and $f V_{\dec{a^i, b^i} t}^+ \hookrightarrow f V_{\dec{a^j, b^j} t}^-$.
    Then, 
    \[
        \big(fV_{\dec{a^i, b^i} t}^+ \cap U_{Jt}^+\big) \cap \big(fV_{\dec{a^j, b^j} t}^- \cap U_{Jt}^-\big) = fV_{\dec{a^i, b^i} t}^+ \cap U_J^-
    \]
        and $fV_{\dec{a^i, b^i} t}^+ \cap U_{Jt}^+ \hookrightarrow fV_{\dec{a^j, b^j} t}^- \cap U_{Jt}^+$,
    so there are inclusions $B_{\dec{a^i, b^i} t} {\hookrightarrow} A_{\dec{a^j, b^j} t}$. Thus, $\{(A_{\dec{a^i, b^i}t}, B_{\dec{a^i, b^i}t}):i\in\Gamma\}$
    is a disjoint set of
    sections of $U_{Jt}$.  
\end{proof}

The contraposition of the previous result shows the already mentioned fact that, when there are nested intervals, $\cM_f$ may not induce a partial matching.

\begin{corollary}
  If, for a given set of intervals $S\subseteq S_V$, we have that
  \[
    \sum_{I \in S} \cM_{f}(I,J) > n_J,
  \]
  then there are at least two nested intervals in $S$.
\end{corollary}

\subsection{Inducing a partial matching}\label{sec:inducing}

Although $\cM_f$ does not
always induce a partial matching, we can use it to obtain a new block function, $\widetilde{\cM}_f$,
which does.
 Firstly, let $\mathcal{J}= \big\{J \in S_U : \sum_{I\in S_V} \cM_f(I,J) > n_J \big\}$. 
 For each $J \in \mathcal{J}$,
 we create the multiset
\[
    \mtb_J = \big(S_J, \widetilde{m}\,\big) = \big\{\big(I, \widetilde{m}_I\big) :\, I \in S_V,\text{ with } \widetilde{m}_I = \cM_f(I,J) \neq 0 \big\}\ ,
\]
and calculate a partial matching, 
\[
    \sigma_J : \rep \mtb_J \rightarrow \rep \{J, n_J\}\ ,
\]
depending on the application. For
example, one could consider
$\sigma_J$ that provides
the bottleneck distance between $\mtb_J$ and $\{(J, n_J)\}$.
Once we have a partial matching $\sigma_J$ for each $J$,
we define the new block function as:
\[
    \widetilde{\cM}_f(I,J) =
    \left\{
    \begin{array}{c l}
        \#\big\{ i : \, \exists j \text{ with } \sigma_J(I_i) = J_j \big\} & \text{ if } J \in \mathcal{J} \text{ and } I \in S_J \\
        \cM_f(I,J) & \text{ otherwise }
    \end{array}
    \right.
\]

Note that, by definition, $\sum_{I\in S_V} \widetilde{\cM}_f(I,J) \leq n_J$ and, since $\widetilde{\cM}_f(I,J) \leq \cM_f(I,J)$, we have that $\sum_{J\in S_U} \widetilde{\cM}_f(I,J) \leq m_I$ by Theorem~\ref{the:well}.
Then, we can induce directly a partial matching using $ \widetilde{\cM}_f$.
\begin{example}
    Let us consider a pair of persistence morphisms
    $f,g:V \to U$
    where $f$ is given by the following commutative diagram:
    \[
    \begin{tikzcd}
    0
    \arrow[r] & 
       0
       \arrow[r] &  0 \arrow[r] & 
          0         \\
    k
    \arrow[r]\arrow[u] &   k^2
    \arrow[r]\arrow[u] & k \arrow[r]\arrow[u]
    & 
      0
    \arrow[u]
    \end{tikzcd}
    \oplus
    \begin{tikzcd}
    k
    \arrow[r] & 
       k
       \arrow[r] &  0 \arrow[r] & 
          0         \\
    k
    \arrow[r]\arrow[u] &   k
    \arrow[r]\arrow[u] & k \arrow[r]\arrow[u]
    & 
      k
    \arrow[u]
    \end{tikzcd}
    \]
    and $g$ is given by:
    \[
    \begin{tikzcd}[/tikz/column 4/.style={column sep=-0.5em}]
    k
    \arrow[r] & 
    k
    \arrow[r] &  
    0 
    \arrow[r] & 
    0         
    \\
    k^2
    \arrow[r, "\sbm{ 1 \; 0 \\ 0 \; 0 \\ 0 \; 1  }"]
    \arrow[u, "\sbm{ 1 \; 1}"] &   
    k^3
    \arrow[r, "\sbm{ 0 \; 1 \; 0 \\ 0 \; 0 \; 1}"]
    \arrow[u, "\sbm{ 1 \; 1\; 1}", swap] & 
    k^2 
    \arrow[r, "\sbm{ 0 \; 1}"]\arrow[u] & 
    k
    \arrow[u] & .
    \end{tikzcd}
    \]
    Note that the image of $f$ and $g$ is the same, then, by \cite[Prop.~5.4]{induced1},
    $\chi_f$ is equal to  $\chi_g$. In particular, it is the same matching as the one induced by $\cM_f$,
    $
        [1,4] \mapsto [1,2].
    $
    However, $\cM_g$ gives the following non-null values,
    \[
        \cM_g([1,4], [1,2]) = 1, 
        \qquad 
        \cM_g([2,3], [1,2]) = 1.
    \]
    We have two options: either match
    $[1,4]$ with $[1,2]$
    or $[2,3]$ with $[1,2]$.
    If we choose the first option, we obtain the same partial matching as the one induced by ${\cal M}_f$. If we choose the second option by  matching the longest bars, then we get a different partial matching.
\end{example}

\section{A matrix method for computing \texorpdfstring{$\cM_f$}{Mf} \texorpdfstring{for p.f.d. persistence modules}{for p.f.d. persistence modules}}
\label{sec:matrix}

The main goal in this section is to provide a combinatorial method, based on matrix column additions, to obtain the block function $\cM_f$.
To this aim, we limit ourselves to persistence modules of finite vectors spaces
indexed by $\R$, also known as \emph{pointwise finite-dimensional} (p.f.d.) persistence modules~\cite{decomposition}.

Let $I=\dec{a,b}$ and $J=\dec{c,d}$ be intervals of $\R$.
We assume that $c \leq a \leq d \leq b$ (otherwise, $X_{IJ} = 0$).
Let $f: V \rightarrow U$ be a
morphism  between persistence modules together with a pair of persistence bases $\cA = \{ \alpha_i \}_{i \in \Lambda}$ for $V$ and $\cB = \{ \beta_i \}_{i \in \Gamma}$ for $U$.
For a fixed $t$, let us focus on the following composition of linear maps:
\begin{equation}\label{diag:induced_f_abcd}
\begin{tikzcd}[column sep=1.5cm]
    V^+_{It} \arrow[r, hookrightarrow, "\iota_{It}"] &
    V_t \arrow[r, "f_t"] &
    U_t \arrow[r, twoheadrightarrow, "\pi_{Jt}"] &
    U_t \big/ U^-_{Jt} \ .
\end{tikzcd}
\end{equation}

By Proposition~\ref{prop:A-plus-minus}, $\cA^+_{It}$ is a persistence basis for $V^+_{It}$ and $\cB_t \setminus \cB^-_{Jt}$ is a persistence basis for $U_t \big/ U^-_{Jt}$
using also  Lemma~\ref{lem:quotient-sub-base}. 
Therefore, we consider the associated matrix ${\cal L}_{IJt}$ of the composition~(\ref{diag:induced_f_abcd}) on the bases $\cA^+_{It}$ and $\cB_t \setminus \cB^-_{Jt}$:
\begin{equation}
\label{matrix:M^f_(ab)t}
{\cal L}_{IJt} :=
\left(
\begin{array}{c|cc}
 &
\cA^-_{It} &
\cA_{It}
\\
\hline
\cB_{Jt} &
1 \cellcolor{blue!20} &
2 \cellcolor{green!20}
\\
\cB_t \setminus \cB^+_{Jt} &
*  &
*
\end{array}
\right)
\end{equation}

We define the reduced matrix ${\cal N}_{IJt}$ that one obtains after a
Gaussian elimination of ${\cal L}_{IJt}$ by using left to right column additions.
Then we consider the following submatrices of ${\cal N}_{IJt}$:
\begin{itemize}
  \item $\cR^+_{IJt}$ := matrix restricted to the rows of ${\cal N}_{IJt}$ associated to $\cB_{Jt}$ and the columns from $\cA^+_{I} = \cA_{I} \cup 
  \cA^-_{I}$ which are zero on the rows associated to $\cB_t \setminus \cB^+_{Jt}$.
  \item $\cR^-_{IJt}$ := matrix restricted to the rows of ${\cal N}_{IJt}$ associated to $\cB_{Jt}$ and the columns from $\cA^-_{I}$ which are zero on the rows associated to $\cB_t \setminus \cB^+_{Jt}$.
  \item $\cR_{IJt}$ := matrix restricted
  to the rows of ${\cal N}_{IJt}$ associated to $\cB_{Jt}$ and the columns from $\cA_{I}$ which are zero on the rows associated with $\cB_t \setminus \cB^+_{Jt}$.
\end{itemize}

The submatrix $\cR^+_{IJt}$ is contained within the block regions \colorbox{blue!20}{$1$} and \colorbox{green!20}{$2$} from 
Expression~(\ref{matrix:M^f_(ab)t}), while $\cR^-_{IJt}$ is contained within the block \colorbox{blue!20}{$1$} and $\cR_{IJt}$ within the block \colorbox{green!20}{$2$}.
Let $\biggen{\cR^{\pm}_{IJt}}$ and $\biggen{\cR_{IJt}}$ the subspaces of
$\gen{\cB_{Jt}}$ which are generated by the columns of the respective matrices.

To find $X_{IJt}$, we use the following quotients
\[
X^+_{IJt} =
\dfrac{fV^+_{It} \cap U^+_{Jt} + U^-_{Jt}}{
  U^-_{Jt}} \hspace{0.2cm} \mbox{ and } \hspace{0.2cm}
X^-_{IJt} =
\dfrac{fV^-_{It} \cap U^+_{Jt} + U^-_{Jt}}{
  U^-_{Jt}}\ ,
\]
\noindent
which satisfy,
\[
X_{IJt} =
\dfrac{fV^+_{It} \cap U^+_{Jt}}{
  fV^-_{It} \cap U^+_{Jt} + fV^+_{It} \cap U^-_{Jt}}
\simeq
\dfrac{fV^+_{It} \cap U^+_{Jt} + U^-_{Jt}}{
  fV^-_{It} \cap U^+_{Jt} + U^-_{Jt}} \simeq \dfrac{X^+_{IJt}}{X^-_{IJt}}\ .
\]

In the following, we work with the subindex sets $\Lambda^+_{It} \subseteq \Lambda$ and $\Gamma_{Jt} \subseteq \Gamma$ so that $\cA^+_{It} = \{ \alpha^1_{it}\}_{i \in \Lambda^+_{It}}$ and also
$\cB_{Jt} = \{ \alpha^1_{it}\}_{i \in \Gamma_{Jt}}$.

\begin{proposition}~\label{prop:X-plus-minus-abcd}
For all $t \in I\cap J$, the following equalities hold:
\begin{enumerate}[label=\emph{\alph*})]
\item\label{XR-plus}   $X^+_{IJt} = \biggen{\cR^{+}_{IJt}}$\ ,
\item\label{XR-minus}  $X^-_{IJt} = \biggen{\cR^{-}_{IJt}}$\ .
\end{enumerate}
\end{proposition}

\begin{proof}
Let us prove first~(\ref{XR-plus}), and, in particular, the inclusion $\supseteq $.
Consider $\gamma \in \cR^{+}_{IJt}$.
By construction, we must have $\gamma = \pi_{Jt} \circ f_t \circ \iota_{It} \big(\sum_{i \in \Lambda^+_{It}} x_i \alpha^1_{it} \big)$ for some coefficients $x_i \in k$ for all $i \in \Lambda^+_{It}$.
This implies $\gamma \in \big(fV^+_{It} + U^-_{Jt}\big) \big/ U^-_{Jt}$.
Also, by hypotheses, there exist coefficients $
x'_i \in k$ for all $i \in \Gamma_{Jt}$, so that
$\gamma = \sum_{i \in \Gamma_{Jt}} 
x'_i \beta^1_{it}$ and so $\gamma \in U_{Jt}$.
Altogether, $\gamma \in X^+_{IJt}$ and thus the claim follows since $X^+_{IJt}$ is a well-defined subspace of 
$U_{Jt}$.
\\
Let us show now that the inclusion $\subseteq$ from~(\ref{XR-plus}) holds.
Consider $\sigma \in X^+_{IJt}$. Since $\sigma \in U_{Jt}$, it can be written in terms of the persistence basis
$\cB_{Jt}$ and so we might consider $\sigma$ as a column $C_\sigma$ of coordinates in $\cB_t \setminus \cB^-_{Jt}$ whose coordinates in $\cB_t \setminus \cB^+_{Jt}$ are all zero.
On the other hand, since $\sigma \in \big(fV^+_{It} + U^-_{Jt}\big) \big/ U^-_{Jt}$, we
can
write $C_\sigma$ as a combination of columns from
${\cal L}_{IJt}$.
Since elementary column operations preserve rank, the reduced matrix 
${\cal N}_{IJt}$
of ${\cal L}_{IJt}$
still generates $C_\sigma$ by the Rouch\'e-Frobenius theorem.
In particular, since the coordinates of $C_\sigma$ in $\cB_t \setminus \cB^+_{Jt}$ are all zero, we 
can
write $C_\sigma$ in terms of $\cR^{+}_{IJt}$, and the claim follows.
\\
One might repeat the argument to show that the equality~(\ref{XR-minus}) holds, changing $\cR^{+}_{IJt}$ for $\cR^{-}_{IJt}$ and writing $X^-_{IJt}$ instead of $X^+_{IJt}$ in the above two paragraphs. On the second paragraph, $C_\sigma$ must be written in terms of columns from the block \colorbox{blue!20}{$1$} 
in Expression~(\ref{matrix:M^f_(ab)t}).
\end{proof}

Now, 
we characterize $X_{IJt}$ in terms of a matrix that can be easily computed.
\begin{theorem}\label{the:Rmatrix}
$X_{IJt} \simeq \biggen{\cR_{IJt}}$ for all $t \in I\cap J$.
\end{theorem}

\begin{proof}
It follows from Proposition~\ref{prop:X-plus-minus-abcd} and Lemma~\ref{lem:quotient-sub-base} and the fact that the corresponding sets of columns satisfy $\cR_{IJt} = \cR^{+}_{IJt} \setminus \cR^{-}_{IJt}$.
\end{proof}

Finally, due to Theorem~\ref{the:Rmatrix}, we have the following result.

\begin{corollary}
Assuming that $\cM_f(I, J)$ is equivalent to the dimension of $X_{IJt}$ for some $t\in I\cap J$, then:
\[
    \cM_f(I, J) = \text{ the number of pivots in }  \cR_{IJt}.
\]
\end{corollary}

\begin{example}
Let $k$ be a field of characteristic $\neq 2$, e.g. $k = \mathbb{Z}_3$.
Consider 
$V \simeq k_{[1,4]} \oplus k_{[2,3]} \oplus k_{[2,5]}$ and $U \simeq k_{[0,3]} \oplus k_{[1,4]}$. We take the canonical bases $\cA=\{\alpha_1, \alpha_2, \alpha_3\}$ for $V$ and $\cB=\{\beta_1, \beta_2\}$ for $U$, see below. 
\begin{center}
\vspace{0.3cm}
    \begin{tikzpicture}[yscale=0.6]
        \draw[->, line width=0.1em] (-0.5,-0.5)--(11,-0.5);
          \foreach \i [evaluate=\i as \x using \i*2] in {0,1,2,3,4,5}{
            \draw[color=gray] (\x, -0.6)--(\x, 3);
            \draw[line width=0.1em] (\x, -0.7)--(\x, -0.3);
            \node at (\x, -1) {\i};
        }
        \node at (-0.5,2.5) {$\beta_1$};
        \draw[line width=0.2em,] (0,2.5) -- (6,2.5);
        \fill (0,2.5) ellipse (0.24em and 0.4em);
        \filldraw[draw=black, fill=black] (6,2.5) ellipse (0.24em and 0.4em);
        \node at (1.5,2) {$\beta_2$};
        \draw[line width=0.2em] (2,2) -- (8,2);
        \fill (2,2) ellipse (0.24em and 0.4em);
        \filldraw[draw=black, fill=black] (8,2) ellipse (0.24em and 0.4em);
        \node at (1.5,1) {$\alpha_1$};
        \draw[line width=0.2em] (2,1) -- (8,1);
        \fill (2,1) ellipse (0.24em and 0.4em);
        \filldraw[draw=black, fill=black] (8,1) ellipse (0.24em and 0.4em);
        \node at (3.5,0.5) {$\alpha_2$};
        \draw[line width=0.2em] (4,0.5) -- (6,0.5);
        \fill (4,0.5) ellipse (0.24em and 0.4em);
        \filldraw[draw=black, fill=black] (6,0.5) ellipse (0.24em and 0.4em);
        \node at (3.5,0) {$\alpha_3$};
        \draw[line width=0.2em] (4,0) -- (10,0);
        \fill (4,0) ellipse (0.24em and 0.4em);
        \filldraw[draw=black, fill=black] (10,0) ellipse (0.24em and 0.4em);
        \draw[dashed] (-1, 1.5) -- (11.5, 1.5);
        \node (U) at (-1.5, 2.5) { $U$};
        \node (V) at (-1.5, 0) { $V$};
    \end{tikzpicture}
\end{center}
Next,  consider a morphism $f:V\rightarrow U$ between persistence modules 
given by
the following commutative diagram:
\begin{equation*}		
\begin{tikzcd}
			U \\
			V \arrow[u, "f"]
		\end{tikzcd}
		\simeq 
		\begin{tikzcd}
			k \arrow[r, "\sbm{1 \\ 0 }"] & k^2 \arrow[r, "\Id"] &  k^2 \arrow[r,"\Id"] & k^2 \arrow[r,"\sbm{ 0 \; 1 }"] & k \arrow[r] &0 \\
		0 \arrow[r]  \arrow[u, swap] &	k \arrow[r,swap, "\sbm{1 \\ 0\\0 }"]  \arrow[u, swap, "\sbm{ 1 \\ 1 }"] &k^3 \arrow[r,swap, "\Id"]  \arrow[u, swap,"\sbm{1\;1\;2\\ 1\;0\;1}"] &k^3 \arrow[r,swap, "\sbm{1\;0\;0\\ 0\;0\;1 }"]  \arrow[u, swap,"\sbm{1\;1\;2\\ 1\;0\;1 }"] & k^2 \arrow[r,swap, "\sbm{ 0 \; 1 }"] \arrow[u, swap, "\sbm{ 1 \; 1 }"] & k \arrow[u]
		\end{tikzcd}
	\end{equation*}
It leads to the matrices ${\cal L}_{IJt}$
for all $I\in S_V$ and all $J=[s,t]\in S_U$:
$$
{\cal L}_{[1,4][0,3]3} =
\sbm{1 \\ 1}\ , \hspace{1cm} 
{\cal L}_{[2,3][0,3]3} = \sbm{ 1 \\ 0} \ , \hspace{1cm} 
{\cal L}_{[2,5][0,3]3} = \sbm{ 1 \; 1 \; 2 \\ 1 \; 0 \; 1} \ ,
$$
$$
{\cal L}_{[1,4][1,4]4} =
\sbm{1}\ , \hspace{1cm} 
{\cal L}_{[2,3][1,4]4} = \emptyset \ , \hspace{1cm} 
{\cal L}_{[2,5][1,4]4} = \sbm{ 1 \; 1} \ .
$$
Next, we reduce ${\cal L}_{[2,5][0,3]3}$ and ${\cal L}_{[2,5][1,4]4}$, so that we obtain:
$$
{\cal N}_{[2,5][0,3]3} = \sbm{ 1 \; 1 \; 0 \\ 1 \; 0 \; 0} \ , \hspace{1cm}
{\cal N}_{[2,5][1,4]4} = \sbm{ 1 \; 0} \ .
$$
Altogether, we obtain:
$$
{\cal R}_{[1,4][0,3]3} = 
{\cal R}_{[2,3][1,4]4} 
= 
\emptyset \ , \hspace{1cm}
{\cal R}_{[1,4][1,4]4} 
= 
{\cal R}_{[2,3][0,3]3} 
= \sbm{ 1 }\ ,
$$ 
$$
{\cal R}_{[2,5][0,3]3} = {\cal R}_{[2,5][1,4]4}
=
\sbm{ 0 }\ .
$$
Hence, a partial matching that can be obtained from $\cM_f$ is:
\[
[1,4] \mapsto [1,4]\;\mbox{ and }\;
[2,3]\mapsto [0,3]\ .
\]
Notice that, in this example, $\Ima(f) \simeq k_{[1,4]} \oplus k_{[2,3]}$ and so the induced partial matching ${\cal X}_f$ is different from ${\cal M}_f$, as is given by
$$[1,4] \mapsto [1,4]\;\mbox{ and }\;
[2,5]\mapsto [0,3]\ .$$
\end{example}

\section{Conclusions and future work}\label{sec:future}

\indent In this paper, we have provided an entirely algebraic definition of
a 
block function
$\cM_f$ induced by a morphism between two persistence modules $V$ and $U$.
We have
proven that 
$\cM_f$ is linear with respect to the direct sum of morphisms.
We have also discussed how to derive partial matchings from  $\cM_f$ and provided a combinatorial method to compute $\cM_f$ based on matrix operations.

We are developing an efficient algorithm to compute partial matchings from morphisms between persistence modules, taking Section~\ref{sec:matrix} as a starting point.
A specially interesting case is the morphism obtained from a function between point clouds
and the corresponding 
Vietoris-Rips filtrations.
Moreover, since in some special cases
the proposed induced partial matching is defined ad-hoc from $\cM_f$ and not algebraically, we are investigating  if it is possible to 
find
an alternative algebraic definition of $\cM_f$ such that $\cM_f$ always gives directly a partial matching.

Some additional research directions can be followed from this paper, for example, the relation between $\cM_f$ and indecomposable modules.
We believe that, when $\sum_{I} \cM_f(I,J) > n_J$,
there should exist non-trivial indecomposable modules containing $J$.
Finally,
relations between persistence modules which come from dynamical systems are
not usually given by morphisms but
through diagrams of the form $V \leftarrow W \rightarrow U$ \cite{sampled, self}.
Then, another research line could be to construct block functions in that context.

\paragraph{\textbf{Acknowledgements and Funding}}
The authors
would like to thank Victor Carmona for his valuable help during the development of this research.
This research was funded by 
the Agencia Estatal de Investigación/10.13039/501100011033
grant PID2019-107339GB-100 and the Agencia Andaluza del Conocimiento grant P20-01145.
The author M. Soriano-Trigueros is partially funded by the grant VI-PPITUS from University of Seville.

\bibliographystyle{abbrv} 
\bibliography{references}


\begin{appendices}


\section{Direct Limits}\label{app:direct}


Let $\cJ$ be filtered category, and consider a functor $F:\cJ \rightarrow \Vect_k$. We consider the direct limit $\varinjlim_{j\in \cJ} F \in \Vect_k$, which is a particular case of the general definition of colimits ( see~\cite[Sec.~3.1.]{riehl2017category}
and \cite[Sec.~3.]{working}). Since $\Vect_k$ is an abelian category, $\varinjlim_{i\in \cJ} F$ exists.
Examples of colimits include direct sums of vector spaces, $\bigoplus_{j \in \cJ} V_i$, and cokernels, $\coker(f)$, of linear maps $f:A\rightarrow B$. 
By~\cite[Thm.~3.8.1]{riehl2017category}, colimits commute with colimits, deducing Lemma~\ref{app:commute-direct-sum} below.
\begin{lemma}\label{app:commute-direct-sum}
Consider a set of functors $\{ F_i:\cJ \rightarrow \Vect_k\}_{i \in \Gamma}$, then:   
    \[
    \varinjlim_{j \in \cJ} \bigoplus_{i \in \cI} F_i(j) \simeq \bigoplus_{i \in \cI} \varinjlim_{j \in \cJ} F_i(j)\ ,
    \]
\end{lemma}

\begin{lemma}\label{app:commute-quotients}
Consider a pair of functors $F_1, F_2:\cJ\rightarrow \Vect_k$ such that $F_1(j)\subseteq F_2(j)$ for all $j \in \cJ$. Thus:
\[
    \varinjlim_{j \in \cJ} \left( F_2(j) \big/ F_1(j) \right) \simeq \varinjlim_{j \in \cJ} F_2(j) \big/\varinjlim_{j \in \cJ} F_1(j) \ . 
\]
\end{lemma}

\begin{proof}
This holds since $F_2(j)\big/ F_1(j)$ is a cokernel, and so a colimit.
\end{proof}

\begin{lemma}\label{lem:short}
	Let U, V and W be persistence modules indexed by an interval $\dec{a,b}$, and let $d\in\E$ with
	$a <d \leq b$.
    A short exact sequence of persistence modules
	\[
		0 \rightarrow V \rightarrow U \rightarrow W \rightarrow 0\ ,
	\]
	produces the following short exact sequence of vector spaces:
	\[
	0 \rightarrow \varinjlim_{t \in \dec{a,d}} V_t \rightarrow \varinjlim_{t \in \dec{a,d}} U_t \rightarrow \varinjlim_{t \in \dec{a,d}} W_t \rightarrow 0\ .
	\]
	In particular, if 
	$V_t \hookrightarrow U_t$ for all $t\in\dec{a,d}$, then 
	\[
	    \varinjlim_{t \in \dec{a,d}} V_t \hookrightarrow \varinjlim_{t \in \dec{a,d}} U_t\ .
    \]
\end{lemma}

\begin{proof}
    Since colimits are the categorical definition of direct limits,  $\Vect_k$ is an abelian category and any totally ordered set is a filtered category, then
    the first result is a direct consequence of the characterization of abelian categories 
    (see 
    \cite[Appendix A.4]{homological}; the original result comes from \cite{abelian}).
    The second
    result follows directly, since injections can be defined in terms of exact sequences.
\end{proof}

\begin{lemma}\label{lem:monos}
	Consider a persistence module 	$V$ with structure maps $\rho$, indexed by an interval $\dec{a,b}$, and 
	let $d\in\E$ such that
	$a <d \leq b$.
	If all structure maps 
	$\rho_{sr}$ with $s,r \in\dec{a,d}$ are injective,
	then
	\[
	    V_s
	    \hookrightarrow
	    \varinjlim_{t \in \dec{a,d}} V_t\,.
    \]
\end{lemma}

\begin{proof}
	Fix $s \in \dec{a,d}$.
    Let us consider the persistence module $C$ as the constant space $V_s$ in $\dec{s^-,d}$ and $0$ in $\dec{a,s^-}$.
	Since all structure maps in $V$ are injective,
	we have
	$
		C \hookrightarrow V.
	$
	Then
	by Lemma~\ref{lem:short},
    \[
		\varinjlim_{t \in \dec{s^-,d}} C_t
		\hookrightarrow \varinjlim_{t \in \dec{s^-,d}}
		V_t.
	\]
	and,  by definition,
	\[
	\varinjlim_{t \in \dec{s^-,d}} C_t
	= V_s
\;	\mbox{ and }\;
	\varinjlim_{t \in \dec{s^-,d}} V_t
	=
	\varinjlim_{t \in \dec{a,d}} V_t,
	\] 
	concluding the proof.
\end{proof}

\begin{lemma}\label{lem:lim-sum-int-commute}
Let $\cJ$ be a filtered category and consider three functors from $\cJ$ to $\Vect_k$ which we call $A$, $B$ and $C$. Further, suppose that $A(j), B(j)$ are both subspaces of $C(j)$ for all $j \in \cJ$. Then
\begin{enumerate}[label=\emph{\alph*})]
    \item\label{lim-inters-commute} $\varinjlim_{j \in \cJ} \big(A(j) \cap B(j)\big) \simeq \left(\varinjlim_{j \in \cJ} A(j) \right) \cap \left(\varinjlim_{j \in \cJ} B(j)\right)\ $, and 
    \item\label{lim-sum-commute} $\varinjlim_{j \in \cJ} \big(A(j) + B(j)\big) \simeq \left(\varinjlim_{j \in \cJ} A(j) \right) + \left(\varinjlim_{j \in \cJ} B(j)\right)\ $.
\end{enumerate}
\end{lemma}

\begin{proof}
First, as $\Vect_k$ is an abelian category,
     filtered direct limits are exact (see~\cite[Appendix A.4]{homological}), and so they commute with kernels and cokernels.
We proceed to prove~(\ref{lim-inters-commute}). For each $j \in \cJ$, consider the exact sequence
\begin{equation}\nonumber
    \begin{tikzcd}[column sep=0.5cm]
        0 \ar[r] & 
        A(j) \cap B(j) \arrow[r, "\iota_j", near start] & 
        A(j) \oplus B(j) \arrow[r, "\Sigma_j", near start] & 
        C(j) 
    \end{tikzcd}
\end{equation}
where $\iota_j$ sends $v \in A(j) \cap B(j)$ to $(v,-v) \in A(j) \oplus B(j)$ and $\Sigma_j$ sends $(v,w) \in A(j) \oplus B(j)$ to $v+w \in C(j)$. Since $A(j) \cap B(j) \simeq \ker(\Sigma_j)$ for all $j \in \cJ$, 
the isomorphism~(\ref{lim-inters-commute}) follows from using
Lemma~\ref{app:commute-direct-sum}. That is,
\begin{multline*}
\varinjlim_{j\in \cJ}\Big( A(j) \cap B(j) \Big) \simeq \varinjlim_{j\in \cJ} \ker \Big(A(j) \oplus B(j) \rightarrow  C(j) \Big) \\
\simeq
\ker\left( \varinjlim_{j\in \cJ} A(j) \oplus \varinjlim_{j\in \cJ}B(j) \rightarrow \varinjlim_{j\in \cJ}  C(j)\right)
\simeq 
\bigg(\varinjlim_{j \in \cJ} A(j)\bigg) \cap  \bigg(\varinjlim_{j\in \cJ}B(j)\bigg)\ .
\end{multline*}
Next, we prove~(\ref{lim-sum-commute}). For each $j \in \cJ$, consider the short exact sequence
\begin{equation}\nonumber
    \begin{tikzcd}[column sep=0.5cm]
        0 \ar[r] & 
        A(j) \cap B(j) \arrow[r, "\iota_j", near start] & 
        A(j) \oplus B(j) \arrow[r, "\Sigma_j", near start] & 
        A(j) + B(j) \ar[r] & 0\ .
    \end{tikzcd}
\end{equation}
Since $A(j)+B(j) \simeq \coker(
\iota_j
)$ for all $j \in \cJ$, 
(\ref{lim-sum-commute}) follows by 
Lemma~\ref{app:commute-direct-sum} together with part~(\ref{lim-inters-commute}). That is,
\begin{multline*}
\varinjlim_{j\in \cJ}\big( A(j) + B(j) \big) \simeq 
\varinjlim_{j\in \cJ} \Big(\coker\big(A(j)\cap B(j) \hookrightarrow A(j) \oplus B(j) \big)\Big) \\
\simeq 
\coker\bigg(\varinjlim_{j\in \cJ} \big(A(j)\big)\cap \varinjlim_{j\in \cJ} \big(B(j)\big) \hookrightarrow \varinjlim_{j\in \cJ} \big(A(j)\big) \oplus \varinjlim_{j\in \cJ} \big(B(j)\big)\big) \bigg) \\
\simeq \varinjlim_{j\in \cJ}\big( A(j)\big) + \varinjlim_{j\in \cJ}\big( B(j)\big)\ . \tag*{\qedhere}
\end{multline*}
\end{proof}

\section{Technical lemmas and proofs related to persistence
bases}\label{app:persistence-basis}

\begin{proof}[Proof of Lemma~\ref{lem:Im-Ker-plus-minus}]
First we prove~(\ref{Im-plus-minus}), i.e. $\Ima^{\pm}_{ct}(V) = \langle \cI^{\pm}_{ct}(\cA)\rangle$ for all $t \in \interval{c}{\infty}$.
We start by showing that the inclusion $\supseteq$ holds.
Given $\alpha^1_{it} \in \cI^+_{ct}(\cA)$ (resp. $\alpha^1_{it} \in \cI^-_{ct}(\cA)$), we have that $\rho_{st}(\alpha^1_{is}) = \alpha^1_{it}$ for all $s \in \dec{c, t^+}$ (resp. for some $s \in \dec{-\infty, c}$). In particular, $\alpha^1_{it} \in \Ima^+_{ct}(V)$
(resp. $\alpha^1_{it} \in \Ima^-_{ct}(V)$ ). Since $\Ima^\pm_{ct}(V)$ are well-defined subspaces of $V_t$, we obtain $\Ima^+_{ct}(V) \supseteq \gen{ \cI^+_{ct}(\cA)} $ (resp. $\Ima^-_{ct}(V) \supseteq \gen{ \cI^-_{ct}(\cA)} $).
\\
Let us now show that $\subseteq$ holds. Consider $v \in \Ima^\pm_{ct}$.
Let any $s \in \dec{c, t^+}$ (resp. some $s \in \dec{-\infty, c}$)
such that there exists $w \in V_s$ with $\rho_{st}(w)=v$.
Since $\cA_t = \{\alpha^1_{it}\}_{i \in \Lambda(t)}$ is a persistence basis for $V_t$, there exists a subindex set
$\Gamma \subseteq \Lambda(t)$, together with coefficients $x_i
\in k \setminus \{0\}$ for all $i \in \Gamma$ such
that $\sum_{i \in \Gamma} x_i \alpha^1_{it} = v$.
Similarly, there exists a subset
$K \subseteq \Lambda(s)$, together with coefficients $x'_i
\in k \setminus \{0\}$ for all $i \in K$ such
that $\sum_{i \in K} x'_i
\alpha^1_{is} = w$.
Altogether, we obtain:
\[
\rho_{st}(w)= \sum_{i \in K} x'_i
\rho_{st}(\alpha^1_{is})
= \sum_{i \in K \cap \Lambda(t)} x'_i
\alpha^1_{it} = v =
\sum_{i \in \Gamma} x_i
\alpha^1_{it}\ ,
\]
which, by linear independence of $\cA_t$, implies that 
$
K \cap \Lambda(t) = \Gamma$ and also $x'_i= x_i
$ 
for all $i \in \Gamma$.
Hence, we have that given $i \in \Gamma$ with $\alpha_i \sim \dec{a_i, b_i}$, we must have $s \in \dec{a_i, b_i}$.
In particular, since we picked up any $s \in \dec{c, t^+}$ (resp. some $s \in \dec{-\infty, c}$), we have that
$\alpha_i \in \cI^+_{c}(\cA)$ (resp. $\alpha_i \in \cI^-_{c}(\cA)$) for all $i \in \Gamma$. This implies that
$v \in \gen{\cI^+_{ct}(\cA)}$ (resp. $v \in \gen{\cI^-_{ct}(\cA)}$) as claimed.
\\
The proof of~(\ref{Ker-plus-minus}) is analogous to that of~(\ref{Im-plus-minus}), although for completeness we reproduce it here. 
That is, we are going to show that  $\Ker^{\pm}_{ct}(V) = \langle \cK^{\pm}_{ct}(\cA)\rangle$ for all $t \in \interval{-\infty}{c}$.
So, let us show first that the inclusion $\subseteq$ holds.
Given $\alpha^1_{it} \in \cK^+_{ct}(\cA)$ (resp. $\alpha^1_{it} \in \cK^-_{ct}(\cA)$), we have that $\rho_{ts}(\alpha^1_{it}) = 0$ for all $s \in \dec{c, \infty}$ (resp. for some $s \in \dec{t^-, c}$).
In particular, $\alpha^1_{it} \in \Ker^+_{ct}(V)$ (resp. $\alpha^1_{it} \in \Ker^-_{ct}(V)$). As $\Ker^\pm_{ct}(V)$ are subspaces of $V_t$, we obtain the inclusions
$\Ker^+_{ct}(V) \supseteq \gen{\cK^+_{ct}(\cA)}$ and $\Ker^-_{ct}(V) \supseteq \gen{\cK^-_{ct}(\cA)}$.
Finally, the inclusion $\subseteq$ from~(\ref{Ker-plus-minus}) follows from Lemma~\ref{lem:subbasis}; notice that in the case $\Ker^-_{ct}(V) \subseteq \gen{\cK^-_{ct}(\cA)}$, if $c$ is decorated by $+$ we need to change the decoration to $-$ to apply Lemma~\ref{lem:subbasis}.
\end{proof}
\begin{lemma}\label{lem:intersection-sub-bases}
Consider a basis $\cW$ for a vector space $W$, together with a pair of subsets $\cS, \cT \subseteq \cW$. Then $\gen{\cS}\cap\gen{\cT}=\gen{\cS \cap \cT}$.
\end{lemma}
\begin{proof}
The inclusion $\supseteq$ is clear, so we only need to prove $\subseteq$.
We use the notation $\cW=\{ w_i \}_{i\in \Gamma}$, $\cS = \{ w_i\}_{i \in \Gamma(\cS)}$ and $\cT = \{ w_i\}_{i \in \Gamma(\cT)}$, where $\Gamma(\cS)$ and $\Gamma(\cT)$ denote the corresponding subindex sets of $\Gamma$.
Consider a vector $v \in \gen{\cS}\cap\gen{\cT}$, then there exist coefficients $x_i
\in k$ for all $i \in \Gamma(\cS)$ and $x'_i
\in k$ for all $j \in \Gamma(\cT)$ so
that $v = \sum_{i \in \Gamma(\cS)}x_i
w_i$ and also $v = \sum_{j \in \Gamma(\cT)}x'_i
w_j$.
We define $x_i
= 0$ for all $i \in \Gamma \setminus \Gamma(\cS)$ and also $x'_i
= 0$ for all $j \in \Gamma \setminus \Gamma(\cT)$.
Altogether, we obtain $\sum_{i \in \Gamma(\cS)}x_i
w_i - \sum_{j \in \Gamma(\cT)}x'_i
w_j = \sum_{i \in \Gamma}(x_i
- 
x'_i
) w_i = 0$, and by linear independence of $\cW$, we  have that $x_i
=x'_i
$ for all $i \in \Gamma$.
In particular, $x_i
\neq 0$ if and only if $x'_i
\neq 0$, and so if $x_i
\neq 0$ then $i \in \Gamma(\cS) \cap \Gamma(\cT)$. Therefore  $v = \sum_{i \in \Gamma(\cS) \cap \Gamma(\cT)}x_i
w_i$ and $v \in \gen{\cS\cap \cT}$ as claimed.
\end{proof}

\begin{lemma}\label{lem:quotient-sub-base}
Consider a basis $\cW$ for a vector space $W$, together with a subset $\cS \subseteq \cW$. Then $\gen{\cW \setminus \cS} \simeq W \big/ \gen{\cS}$.
\end{lemma}

\begin{proof}
Define a linear map $\phi:\gen{\cW \setminus \cS} \rightarrow W \big/ \gen{\cS}$ sending $w \in \gen{\cW \setminus \cS}$ to its class $w + \gen{\cS} \in W \big/ \gen{\cS}$. It is clear that $\phi$ is surjective.
On the other hand, 
$\dim(\gen{\cW \setminus \cS}) = \dim(W \big/ \gen{\cS})$ and so $\phi$ 
must be an
isomorphism.
\end{proof}

\end{appendices}

\end{document}